\documentclass[oneside, book,10 pt]{amsart}
\usepackage{color}
\usepackage{cancel}
\usepackage{amsmath}
\usepackage{amssymb}
\usepackage[normalem]{ulem}
\usepackage{amsfonts}
\usepackage{amsthm}
\usepackage{enumerate}
\usepackage[mathscr]{eucal}
\usepackage{graphicx}
\usepackage[active]{srcltx}
\usepackage{fancyhdr}
\usepackage{verbatim}
\usepackage{fullpage}
\usepackage{hyperref}
\usepackage{caption}
\usepackage{float}

\newtheorem{thm}{Theorem}[section]
\newtheorem*{thm*}{Theorem}
\newtheorem{lemma}[thm]{Lemma}
\newtheorem{corollary}[thm]{Corollary}
\newtheorem{prop}[thm]{Proposition}

\numberwithin{equation}{section}

\title{Markov random fields, Markov cocycles and the 3-colored chessboard }
\author{
Nishant Chandgotia
\and
Tom Meyerovitch
}\address{
Department of Mathematics\\ University of British Columbia,Canada}
\email {nishant@math.ubc.ca}

\address{
Department of Mathematics\\
Ben-Gurion University of the Negev,
Israel }
\email{mtom@math.bgu.ac.il}
\subjclass[2010]{Primary 37A60; Secondary 60J99}
\keywords{Markov random field, Gibbs state, nearest neighbour interaction, shift-invariance, multidimensional, cocycle, pivot property}
\def\F{{\mathcal F}}
\def\A{{\mathcal A}}
\def\N{\mathbb N}

\def \B{\mathcal B}

\def \E{\mathbb E}
\def \R{\mathbb R}
\def \G{\mathcal G}

\def \E1{\mathcal E}
\def \M{\mathcal M}
\def\p{\prime}

\newcommand{\Ht}{\mathit{Ht}^{(d)}}
\newcommand{\Z}{\mathbb{Z}}
\newcommand{\ZD}{\mathbb{Z}^{d}}
\newcommand{\DHt}{\mathit{grad}}

\begin{document}
\Large
\maketitle
\begin{abstract}
The well-known Hammersley-Clifford Theorem states (under certain conditions) that any Markov random field is a Gibbs state for a nearest neighbor interaction. In this paper we study Markov random fields for which the proof of the Hammersley-Clifford Theorem does not apply. Following Petersen and Schmidt we utilize the formalism of cocycles for the homoclinic equivalence relation and introduce ``Markov cocycles", reparametrizations of Markov specifications. The main part of this paper exploits this to deduce the conclusion of the Hammersley-Clifford Theorem for a family of Markov random fields which are outside the theorem's purview where the underlying graph is $\Z^d$. This family includes all Markov random fields whose support is the d-dimensional ``$3$-colored chessboard". On the other extreme, we construct a family of shift-invariant Markov random fields which are not given by any finite range shift-invariant interaction.
\end{abstract}
\section{Introduction}
A Markov random field is a collection of random variables $(x_v)_{v \in V}$ taking values in a set $\A$ and indexed by the vertices of an undirected graph $G=(V,E)$,
satisfying the following conditional independence property: Fix any two finite subsets of vertices $A,B \subset V$ having no edges between them.
Then the corresponding random variables $(x_v)_{v \in A}$ and $(x_v)_{v \in B}$ are independent given $(x_v)_{v \in V \setminus (A \cup B)}$.

The Hammersley-Clifford Theorem states that under certain assumptions any Markov random field is a so-called \emph{Gibbs state} for a \emph{nearest neighbor interaction} (we recall these definitions in Section \ref{section: Markov _spec_and_cocyc}).
The key assumption in the proof of the Hammersley-Clifford Theorem is the existence of a so-called ``safe symbol'' (see Subsection \ref{safe hamcliff} below).
In this paper, we study Markov random fields outside the purview of the safe symbol assumption.


In general the conclusion of the Hammersley-Clifford Theorem does not hold if we drop the safe symbol assumption: J. Moussouris provided examples of Markov random fields on a finite graph which are not Gibbs states for a nearest neighbor interaction \cite{mouss}.
Are there weaker conditions on a Markov random field which guarantee a Hammersley-Clifford-type theorem? In the absence of a nearest neighbor interaction, are there alternative ``natural parametrizations'' for Markov random fields?
Some work along these lines has been carried out in the first-named author's M.Sc. thesis \cite{Chandgotia}. The current paper extends this investigation.

There has been other work along this direction: When the underlying graph is finite, there are algebraic conditions on the support \cite{sturmfel} and conditions on the graph \cite{laur} which guarantee the conclusions of the Hammersley-Clifford Theorem. When the underlying graph is a Cayley graph of $\mathbb{Z}$ the Markov random field is shift-invariant and $\mathcal{A}$ is finite, the conclusions of the Hammersley-Clifford Theorem hold without any further assumptions \cite{Markovfieldchain}. Furthermore, in that setting any Markov random field is a stationary Markov chain.
Even when the underlying graph is a Cayley graph of $\mathbb{Z}$, this conclusion can fail for countable $\mathcal{A}$ (Theorem 11.33 in \cite{Georgii}), or if we drop the assumption of shift-invariance \cite{dob}.
In the same setting, certain mixing conditions guarantee the conclusion even when $\mathcal{A}$ is a general measure space rather than a finite set (Theorems 10.25 and 10.35 in \cite{Georgii}). When the underlying graph is the standard Cayley graph of $\mathbb{Z}^2$ the conclusion fails even in the shift-invariant and the finite $\A$ case (Chapter 5 in \cite{Chandgotia} and Section \ref{section: nonGibb} of this paper).

The Markov random fields which we consider consist of random variables taking values in a finite set $\mathcal{A}$. For most of this paper we focus on Markov random fields for which the graph $G=(V,E)$ is the Cayley graph of $\mathbb{Z}^d$ with respect to the standard set of generators. The natural action of $\mathbb{Z}^d$ on this graph induces an action on $\mathcal{A}^{\mathbb{Z}^d}$ by translations or \emph{shifts}. We study in particular Markov random fields which are invariant to the shift action. Evidently, the support of any such Markov random field is a \emph{shift space}, that is a shift-invariant compact subset of $\mathcal{A}^{\mathbb{Z}^d}$.

We give a short overview of the paper and our main results:

A major portion of this paper establishes the conclusion of the Hammersley-Clifford Theorem for Markov random fields whose support is in a certain one-parameter family of shift spaces $X_r$, where $r \ne 1,4$ is a positive integer. We prove that any Markov random field whose support is $X_r$ is a Gibbs state for a nearest neighbor interaction (Proposition \ref{prop:X_r_Markov_is_non_stationary_gibbs}). Furthermore, the interaction can be chosen shift-invariant if the concerned Markov random field is shift-invariant (Theorem \ref{thm:MRF_on_X_r_is_Gibbs}). For $r \ne 1,4$, the space $X_r \subset \{0,\ldots,r-1\}^{\mathbb{Z}^d}$ consists of all those $x\in \{0,\ldots,r-1\}^{\mathbb{Z}^d}$ for which $x_n -x_m = \pm 1 \mod r$ whenever $\|n-m\|_1=1$. For $r=2$ this space consists of $2$ periodic ``chessboard'' configurations and for $r=3$ this is the $d$-dimensional ``$3$-colored chessboard'': the set of proper graph $3$-colorings of the standard Cayley graph of $\mathbb{Z}^d$.

In Section \ref{section: Markov _spec_and_cocyc} we introduce ``Markov cocycles'' and ``Gibbs cocycles'', a rather general formalism allowing to parameterize Markov specifications. Following \cite{petersen_schmidt1997}, these are logarithms of the Radon-Nikodym derivatives with respect to the homoclinic relation, that is, logarithms of the ratio of probabilities of configurations which differ at only finitely many sites. Markov cocycles form a vector space, the subspace of Gibbs cocycles with shift-invariant interactions being finite dimensional. We formulate a strong version of the Hammersley-Clifford Theorem in terms of these objects (Theorem \ref{thm: strong_hammersley_clifford}). A weaker condition on the support of a Markov random field than that of a ``safe symbol'', namely, the ``pivot property'' suffices to deduce that the space of shift-invariant Markov cocycles itself is already finite dimensional (Proposition \ref{finite pivot dimension}). For $r\in \N \setminus \{1, 2, 4\}$, we prove that the dimension of the shift-invariant Markov cocycles on $X_{r}$ is $r$ while the dimension of Gibbs cocycles on $X_r$ with a shift-invariant interaction is $r-1$ (Propositions \ref{prop:X_r_markov_cocycles} and \ref{prop:Gibbs_cocycles_X_r}). However if we do not demand shift-invariance for the interaction, every Markov cocycle on $X_r$ is Gibbs. Theorem \ref{thm:MRF_on_X_r_is_Gibbs} which is our main result, asserts that the conclusion of a weak version of the Hammersley-Clifford Theorem (Theorem \ref{hamcliff}) holds for $X_r$. The conclusion of a stronger version (Theorem \ref{thm: strong_hammersley_clifford}) holds as well (Proposition \ref{prop:X_r_Markov_is_non_stationary_gibbs}), but not in the shift-invariant setting (Corollary \ref{cor:no_invaraint_interaction}).

In Section \ref{section: nonGibb} we construct a subshift which does not have the pivot property and for which the space of shift-invariant Markov cocycles has uncountable dimension.

\section{Background and Notation}\label{section: Background}
This section will recall the necessary concepts and introduce the basic notation.

\subsection{Markov Random Fields and topological Markov fields}\label{Subsection: MRF and TMF}
Let $G=(V,E)$ be a simple, undirected graph where the vertex set $V$ is finite or countable. We always assume that $G$ is \emph{locally finite}, meaning all $v \in V$ have a finite number of neighbours.
The \emph{boundary} of a set of vertices $F \subset V$, denoted by $\partial F$, is the set of vertices outside $F$ which are adjacent to $F$:
$$\partial F := \left\{ v \in V \setminus F ~:~ \exists w \in F \mbox{ s.t. } (v,w) \in E\right\}.$$

\textbf{Remark:} Observe that in our notation $\partial F \subset F^c$. This is sometimes called the \emph{outer boundary} of the set $F$. Consistent with our notation the \emph{inner boundary} of $F$ is $\partial (F^c)$.

Given a finite set $\mathcal{A}$, 
the space $\A^V$ is a compact topological space with respect to the product topology, where the topology on $\A$ is discrete. For $F \subset V$ finite and $a \in \mathcal{A}^F$, we denote by $[a]_F$ the \emph{cylinder set}$$ [a]_F := \left\{ x \in \mathcal{A}^V~:~ x|_F =a\right\}.$$
For $x \in \mathcal{A}^V$ we use the notation $[x]_F$ for $[x|_F]_F$.
The collection of cylinder sets generates the Borel $\sigma$-algebra on $\mathcal{A}^V$.

A \emph{Markov random field} is a Borel probability measure $\mu$ on $\mathcal{A}^V$ with the property that
for all finite $A, B \subset V$ such that $\partial A \subset B \subset A^{c}$ and $a \in \A^A, b \in \A^B$ satisfying $\mu([b]_B)>0$
\begin{equation*}
\mu\left([a]_A\;\Big\vert\;[b]_B\right)= \mu\left([a]_A\;\Big\vert\;[b|_{\partial A}]_{ \partial A}\right).
\end{equation*}

We say that the sets of vertices $A, B \subset V$ are \emph{separated} (in the graph $G$) if $(v,w) \not\in E$ whenever $v \in A$ and $w \in B$.

Here is an equivalent definition of a Markov random field: If $x$ is a point chosen randomly according to the measure $\mu$, and $A, B \subset V$ are finite and separated, then conditioned on $x|_{V \setminus (A \cup B)}$, $x|_A$ and $x|_B$ are independent random variables.

A Markov random field is called \emph{global} if the conditional independence above holds for all separated sets $A, B \subset V$ (finite or not) .

As in \cite{Chandgotia,Markovfieldchain}, a \emph{topological Markov field} is a compact set $X \subset \A^V$ such that for all finite $F \subset V$ and $x,y\in X$ satisfying $x|_{\partial F}=y|_{\partial F}$ , there exist $z\in X$ satisfying
\begin{eqnarray*}
z_v:=\begin{cases}x_v \text{ for } v \in F\\ y_v\text{ for }v \in V \setminus F.\end{cases}
\end{eqnarray*}

A topological Markov field is called \emph{global} if we do not demand that $F$ be finite. Topological Markov fields which are not global do exist, see  \cite[Chapter $5$]{Chandgotia}.

The \emph{support} of a Borel probability measure $\mu$ on $\mathcal{A}^V$ denoted by $supp(\mu)$ is the intersection of all closed sets $Y \subset \mathcal{A}^{V}$ for which $\mu(Y)= 1$. Equivalently,
$$supp(\mu) = \A^V \setminus \bigcup_{[a]_A \in \mathcal{N}(\mu)}[a]_A,$$
where $\mathcal{N}(\mu)$ is the collection of all cylinder sets $[a]_A$ with
$\mu([a]_A)=0$.
The support of a Markov random field is a topological Markov field (see Lemma $2.0.1$ in \cite{Chandgotia}). 

\subsection{The homoclinic equivalence relation of a TMF and adapted MRFs.}
Following \cite{petersen_schmidt1997, schmidt_invaraint_cocycles_1997}, we denote by $\Delta_X$ the \emph{homoclinic equivalence relation} of a TMF $X \subset \A^{V}$, which is given by
\begin{equation}\label{eq:Homoclinic_relation_def}
\Delta_X := \{(x,y)\in X\times X\;|\; x_n=y_n \text{ for all but finitely many }n\in V\}.
\end{equation}

We say that a Markov random field $\mu$ is \emph{adapted} with respect to a topological Markov field $X$ if
$supp(\mu)\subset X$ and
\begin{equation*}
x\in supp(\mu) \Longrightarrow \{y\in X~:~ (x,y)\in \Delta_X\}\subset supp(\mu).
\end{equation*}

To readers familiar with the theory of non-singular equivalence relations, we remark that a Markov random field $\mu$ is adapted to $X$ if and only if the measure $\mu$ is non-singular with respect to $\Delta_X$. See \cite{Meyerovitch_Gibbs} and references within for more on measures which are non-singular with respect to the homoclinic equivalence relation.

To illustrate this definition, if $supp(\mu) = X$ then $\mu$ is adapted with respect to $X$, and if $X=X_1 \cup X_2$ is the union of two
topological Markov fields over disjoint alphabets and $\mu$ is a Markov random field with $supp(\mu)=X_1$ then $\mu$ is adapted with respect to $X$. On the other hand, the Bernoulli measure $(\frac{1}{2}\delta_0+\frac{1}{2}\delta_1)^{V}$ is \emph{not} adapted with respect to $\{0,1,2\}^{V}$.
In fact, if $X= \A^{V}$ for some finite alphabet $\A$ then any Markov random field which is adapted to $X$ has $supp(\mu)=X$. 

\subsection{$\mathbb{Z}^d$-shift spaces and shifts of finite type}\label{subsection: SFT}
For the Markov random fields we discuss in this paper, the set of vertices of the underlying graph is the $d$-dimensional integer lattice. We identify $\mathbb{Z}^d$ with the set of vertices of the Cayley graph with respect to the standard generators. Rephrasing, $n,m \in \mathbb{Z}^d$ are adjacent iff $\|n -m\|_1 =1$, where
for all $n=(n_1,\ldots,n_d) \in \Z^d$, $||n||_1=\sum_{r=1}^d |n_r|$ denotes
the $l^1$ norm of $n$.
The $boundary$ of a given finite set $F\subset \Z^d$ is thus given by:
\begin{equation*}
\partial F:=\left\{m\in F^c ~:~ ||n-m||_1= 1\text{ for some }n\in F\right\}.
\end{equation*}

On the compact space $\mathcal{A}^{\Z^d}$ (with the product topology over the discrete set $\A$) the maps $\sigma_n:\A^{\Z^d} \to \A^{\Z^d}$ given by
\begin{equation*}
(\sigma_n(x))_m:=x_{m+n} \text{ for all }m,n \in \Z^d
\end{equation*}
define a $\mathbb{Z}^d$-action by homeomorphisms, called the \emph{shift-action}. The pair $(\A^{\Z^d},\sigma)$ is a topological dynamical system called the \emph{d-dimensional full shift} on the alphabet $\A$. Note that $\sigma$ acts on the Cayley graph of $\mathbb{Z}^d$ by graph isomorphisms.

A $\mathbb{Z}^d$-\emph{shift space} or \emph{subshift} is a dynamical system $(X,\sigma)$ where $X\subset \A^{\Z^d}$ is closed and invariant under the map $\sigma_n$ for each $n \in \mathbb{Z}^d$.

A Borel probability measure $\mu$ on $\mathcal{A}^{\mathbb{Z}^d}$ is \emph{shift-invariant}
if $\mu\circ \sigma_n = \mu$ for all $n \in \mathbb{Z}^d$.
It follows that the support of all shift-invariant measures $\mu$ is a subshift.

For $X \subset \A^V$ and $W\subset V$ let
\begin{equation*}
\B_W(X):= \{w\in \A^{W}\;|\;\text{there exists } x\in X \text{ such that } x|_{W}=w\}.
\end{equation*}
The \emph{language }of a set $X \subset \A^{V}$ denoted by $\B(X)$ is defined as all finite configurations which occur in the elements of $X$:
\begin{equation*}
\B(X):= \bigcup_{W\subset V\text{ finite}}\B_W(X).
\end{equation*}

If $A,B \subset V$ and $x \in \A^A, y \in \A^B$ such that $x|_{A\cap B}= y|_{A\cap B}$ then $x \vee y \in \A^{A \cup B}$ is given by
$$ (x \vee y)_n := \begin{cases}
x_n & n \in A\\
y_n & n \in B.
\end{cases}$$

An alternative equivalent definition for a subshift is given by \emph {forbidden patterns} as follows:
Let
\begin{equation*}
\A^\star:= \bigcup_{W\subset \Z^d\text{ finite}}\A^W.
\end{equation*}
For all $\F \subset \A^\star$ let
\begin{equation*}
X_\F=\{x \in \A^{\Z^d}| \text{ no translate of a subconfiguration of } x\text{ belongs to }\F \}.
\end{equation*}
It is well-known that a subset $X\subset \A^{\Z^d}$ is a subshift if and only if there exists $\F \subset \A^\star$ such that $X=X_\F$. The set $\F$ is called the set of forbidden patterns for $X$. A subshift $X$ is called a\emph{ shift of finite type }if $X=X_\F$ for some finite set $\F$. A shift of finite type is called a \emph{nearest neighbor shift of finite type} if $X= X_\F$ where $\F$ consists of nearest neighbor constraints, i.e. $\F$ consists of patterns on single edges. When $d=1$ nearest neighbor shifts of finite type are also called \emph{topological Markov chains}. In fact the study of nearest neighbor shifts of finite type has been partly motivated by the fact that the support of stationary Markov chains are one-dimensional nearest neighbor shifts of finite type.

Every nearest neighbor $\mathbb{Z}^d$-shift of finite type is a shift-invariant topological Markov field. When $d=1$ the converse is also true under the assumption that the subshift is non-wandering \cite{Markovfieldchain}. Without the non-wandering assumption, one-dimensional shift-invariant topological Markov fields are still so-called sofic shifts, but not necessarily of finite type \cite{Markovfieldchain}. This does not hold in higher dimensions (\cite{Chandgotia} and Section \ref{section: nonGibb}).


\subsection{Gibbs states with nearest neighbor interactions}\label{Subsection: Gibbs States}

For a graph $G=(V,E)$ and $A\subset V$, let $diam(A)$ denote the diameter of $A$ with respect to the graph distance (denoted by $d$) in $G$, that is, $diam(A)=\max_{i,j\in A}d(i,j)$.

Following \cite{Rue}, an \emph{interaction} on $X$ is a function $\phi$ from $\B(X)$ to $\R$, satisfying certain summability conditions. Here we will only consider finite range interactions, for which the summability conditions are automatically satisfied.

An interaction is of \emph{range at most $k$} if
$\phi(a)=0$ for $a\in \B_A(X)$ whenever $diam(A)> k$.
We will call an interaction of range $1$ a \emph{nearest neighbor interaction}.
When $G= \mathbb{Z}^d$, an interaction $\phi$ is shift-invariant if for all $n \in \mathbb{Z}^d$ and $a\in \B(X)$, $\phi(a)= \phi(\sigma_n(a))$.
Since the standard Cayley graph of $\Z^d$ has no triangles, a shift-invariant nearest neighbor interaction is uniquely determined by its values on configurations on $\{0\}$ (``single site potentials'') and on configurations on pairs $\{0,e_i\}$, $i=1,\ldots,d$ (``edge interactions''). We denote these by $\phi([a]_0)$ and $\phi([a,b]_i)$ respectively where $a,b \in \A$.

A \emph{Gibbs state with a nearest neighbor interaction $\phi$} is a Markov random field $\mu$ such that for all $x\in supp(\mu)$ and $A, B \subset V$ finite satisfying $\partial A \subset B \subset A^{c}$,
\begin{eqnarray*}
\mu\left([x]_A\;\Big\vert\;[x]_B\right)&=&
\displaystyle{\frac{\prod\limits_{C \subset A \cup \partial A}e^{\phi(x|_C)}}{Z_{A,x|_{\partial A}}} }
\end{eqnarray*}
where $Z_{A,x|_{\partial A}}$ is the uniquely determined normalizing factor so that $\mu(X)=1$. 

\section{Markov specifications and Markov cocycles}\label{section: Markov _spec_and_cocyc}

Any Markov random field $\mu$ yields conditional probabilities of the form $\mu( x_F = \cdot \mid x_{\partial F} = \delta)$ for all finite $F \subset V$ and \emph{admissible} $\delta \in \mathcal{A}^{\partial F}$ (by admissible we mean $\mu([\delta]_{\partial F} )>0$). We refer to such a collection of conditional probabilities as the \emph{Markov specification} associated with $\mu$.
It may happen that two distinct Markov random fields have the same specification, as in the case of the $2$-dimensional Ising model in low temperature \cite{onsager_two_dimension_model_order}. In general it is a subtle and challenging problem to determine if a given Markov specification admits more than one Markov random field (the problem of uniqueness for the measure of maximal entropy of a $\mathbb{Z}^d$-shift of finite type is an instance of this problem \cite{burtonsteiffnonuniquesft}).
For the purpose of our study and statement of our results, it would be convenient to have an intrinsic definition for a Markov specification, not involving a particular underlying Markov random field.

Let $X \subset \mathcal{A}^{V}$ be a topological Markov field.
A \emph{Markov specification} on $X$ is an assignment for each
finite and non-empty $F \subset V$ and $x \in \mathcal{B}_{\partial F}(X)$
of a probability measure $\Theta_{F,x}$ on $\mathcal{B}_F(X)$ satisfying the following conditions:
\begin{enumerate}
\item{\textbf{Support condition}: For all finite and non-empty $F \subset V$, $x \in \mathcal{B}_{\partial F}(X)$ and $y \in \mathcal{B}_{F}(X)$, $x \vee y \in \mathcal{B}_{F \cup \partial F}(X)$ if and only if $\Theta_{F,x}(y)>0$. This condition can be written as follows:
$$\mathit{supp}(\Theta_{F,x}) =\{y \in \mathcal{B}_{F}(X)~:~ x \vee y \in \mathcal{B}_{F \cup \partial F}(X)\}.$$}
\item{\textbf{Markovian condition}: For all finite and non-empty $F \subset V$ and $x \in \mathcal{B}_{\partial F}(X)$, $\Theta_{F,x}$ is a Markov random field on the finite graph induced from $V$ on $F$.}
\item{\textbf{Consistency condition}: If $F \subset H \subset V$ are finite and non-empty, $x \in \mathcal{B}_{\partial F}(X)$, $y \in \mathcal{B}_{\partial H}(X)$ and $x_n=y_n$ for $n \in \partial F \cap \partial H$, then for all $z \in \mathcal{B}_F(X)$
$$ \Theta_{F,x}(z) = \frac{\Theta_{H,y}( [z \vee x]_{(F \cup \partial F)\cap H})}{\Theta_{H,y}([x]_{\partial F \cap H})}.$$}
\end{enumerate}
The definition above has been set up so that for any Markov random field $\mu$ with $X = \mathit{supp}(\mu)$, the assignment $$(F,x) \mapsto \Theta_{F,x}(a):=\mu([a]_F \mid [x]_{\partial F})$$
is a Markov specification. Conversely, given any Markov specification $\Theta$ on $X$ there exists a Markov random field $\mu$ on $X$ compatible with $\Theta$ in the sense that $\mu([a]_F \mid [y]_{\partial F}) =\Theta_{F,y}(a)$ for all $a\in \B_{F}(X)$ whenever $\mu([y]_{\partial F})>0$ (Chapter 4 in \cite{Georgii}). Furthermore, when $X \subset \A^{\ZD}$ is a subshift and the specification $\Theta$ is shift-invariant, it follows from amenability of $\mathbb{Z}^d$ that there exists a shift-invariant Markov random field $\mu$ compatible with $\Theta$.
However, in general it is possible that for a given specification $\Theta$ the support of any $\mu$ satisfying the above is a strict subset of $X$, in which case there exist certain finite $F \subset V$ and $y \in \mathcal{B}_{\partial F}(X)$ for which the conditional probabilities $\mu([x]_F \mid [y]_{\partial F})$ are meaningless for all $x\in X$. We will provide such examples in Section \ref{Section: MRF on 3cb gibbs}. In such a case, according to our definition, the Markov specification associated with $\mu$ is the restriction of $\Theta$ to the support of $\mu$, meaning the collection of conditional probabilities $\Theta_{F,x}$ for finite sets $F\subset V$ and $x\in \B_{\partial F}(supp(\mu))$.

It will be convenient for our purposes to re-parameterize the set of Markov specifications on a given topological Markov field $X$. For this purpose we use the formalism of $\Delta_X$-cocycles. To this well-known formalism we
introduce an ad-hoc definition of a \emph{Markov cocycle}, which we will explain now.
As in \cite{schmidt_invaraint_cocycles_1997}, a (real-valued) \emph{$\Delta_X$-cocycle} is a function $M: \Delta_X\longrightarrow\R$ satisfying
\begin{equation}
M(x,z)=M(x,y) + M(y,z) \mbox{ whenever } (x,y),(y,z) \in \Delta_X.
\end{equation}
We call $M$ a \emph{Markov cocycle} if it is a $\Delta_X$-cocycle and satisfies: For all $(x,y) \in \Delta_X$ such that $x|_{F^c}=y|_{F^c}$
the value $M(x,y)$ is determined by $x|_{F\cup \partial F}$ and $y|_{F \cup \partial F}$. If $G$ is the standard Cayley graph of $\Z^d$, a $\Delta_X$-cocycle is \emph{shift-invariant} if $M(x,y)=M(\sigma_n(x),\sigma_n(y))$ for all $n \in \mathbb{Z}^d$. 

There is a clear bijection between Markov cocycles and Markov specifications on $X$: If $\Theta$ is a Markov specification on $X$, the corresponding Markov cocycle is given by
$$M(x,y) := \log \left( \Theta_{F, y|_{ \partial F}}(y|_F) \right)-\log \left( \Theta_{F, x|_{ \partial F}}(x|_F) \right),$$
where $(x,y)\in \Delta_X$ and $F\subset V$ finite such that $x|_{F^c} = y|_{F^c}$.

Conversely, given a Markov cocycle $M$ on $X$, the corresponding specification $\Theta$ is given by

$$ \Theta_{F,a}(y) = \frac{1}{Z_{M,F,a,z}}e^{M( x \vee z,x \vee y)},$$
where $F \subset V$ is a finite set, $a \in \mathcal{B}_{\partial F}(X)$, $y,z \in \mathcal{B}_{F}(X)$ are such that $a \vee y, a \vee z \in \mathcal{B}_{F \cup \partial F}(X)$ and $x \in \mathcal{B}_{F^c}(X)$ with $x|_{\partial F}=a$. The normalizing constant $Z_{M,F,a,z}$ is given by:

$$Z_{M,F,a,z}= \sum_{ y' }e^{ M( x \vee z, x \vee y^\prime)},$$
where the sum is over all $y^\prime \in \mathcal{B}_{F}(X)$ such that $y^\prime \vee a \in \mathcal{B}_{F \cup \partial F}(X)$.
Note that the expression for the specification is well defined: Since $M$ is a Markov cocycle on the topological Markov field $X$ the right-hand side does not depend on the particular choice of $x$. The choice of the auxiliary variable $z$ on the right-hand side changes the normalizing constant $Z_{M,F,a,z}$, but does not change $\Theta_{F,a}$.

When $X \subset \A^{\ZD}$ is a subshift, this bijection maps shift-invariant specifications to shift-invariant Markov cocycles.
Thus, a shift-invariant Borel probability measure $\mu$ is a shift-invariant Markov random field if and only if $X=supp(\mu)$ is a topological Markov field and there exists a shift-invariant Markov cocycle $M$ on $X$ such that for all $(x,y)\in \Delta_{X}$
\begin{equation*}
\frac{\mu([y]_\Lambda)}{\mu([x]_\Lambda)}=e^{M(x,y)}
\end{equation*}
for all $\Lambda\supset F\cup \partial F$ where $F$ is the set of sites on which $x,y$ differ.

Fix a topological Markov field $X$ and a nearest neighbor interaction $\phi$ on $X$. The \emph{Gibbs cocycle} corresponding to $\phi$ is given by:
\begin{equation*}
M_{\phi}(x,y)=\sum_{W \subset V\text{ finite }} \phi(y|_{W})-\phi(x|_{W}).
\end{equation*}

Note that when $\phi$ is a nearest neighbor interaction, there are only finitely many non-zero terms in the sum whenever $(x,y) \in \Delta_X$ and so $M_{\phi}$ is well defined. Examination of the definitions verifies that under the above assumptions $M_{\phi}$ is a Markov cocycle. Our point of interest is the converse: When can a Markov cocycle be expressed in this form?

\textbf{Remark:} A Borel probability measure $\mu$ is a Gibbs state with nearest neighbor interaction $\phi$ if and only if its $\Delta_X$-Radon-Nikodym cocycle is $e^{M_{\phi}}$ (as in \cite{petersen_schmidt1997}) for the Gibbs cocycle $M_\phi$ on $X=supp(\mu)$. See Section $2$ of \cite{Meyerovitch_Gibbs} and references within.

Let $X$ be a topological Markov field. Denote by $\M_X$ the set of all Markov cocycles and by $\G_X$ the set of all nearest neighbor Gibbs cocycles. In case $X$ is also a subshift, we denote by $\M_X^\sigma$ the set of shift-invariant Markov cocycles and by $\G_X^\sigma$ the set of Gibbs cocycles corresponding to a shift-invariant nearest neighbor interaction.

The set $\M_X$ of Markov cocycles naturally carries a vector space structure: Given $M_1,M_2\in \M_X$ and $c_1,c_2\in \R$,
$
c_1M_1+c_2M_2\in \M_X
$.
The reader can easily verify that $\G_X$ is a linear subspace of $\M_X$.
Since shift-invariant nearest neighbor interactions constitute a finite-dimensional vector space, and the map sending a nearest neighbor interaction $\phi$ to the cocycle $M_\phi$ is linear, it follows that $\G_X^\sigma\subset \M_X$ has finite dimension.

For a topological Markov field $X$ defined over a finite graph $G=(V,E)$, $\mathcal{M}_X$ is finite dimensional; the problem of determining which Markov cocycles are Gibbs amounts to solving a finite (but possibly large) system of linear equations. The resulting equations are essentially the `balanced conditions' mentioned in \cite{mouss}.

\subsection{The ``safe symbol property'' and the Hammersley-Clifford Theorem}\label{safe hamcliff}
A topological Markov field $X\subset \A^{V}$ is said to have a \emph{safe symbol} if there exists an element $\star\in \A$ such that for all $x\in X$ and $A\subset V$, $y \in \A^{V}$ given by
\begin{eqnarray*}
y_n:=
\begin{cases}x_n\text{ for } n \in A\\
\star \ \ \text{ for } n \in A^c\end{cases}
\end{eqnarray*}
is also an element of $X$. 



A formulation of the Hammersley-Clifford Theorem states:

\begin{thm} \textbf{(Hammersley-Clifford, weak version \cite{HammersleyClifford71})}\label{hamcliff}
Let $X$ be a topological Markov field with a safe symbol. Then:
\begin{enumerate}
\item{Any Markov random field with $supp(\mu)=X$ is a Gibbs state for a nearest neighbor interaction.}
\item{\label{enum:stationary_HC} Further if $X$ is a subshift, any shift-invariant Markov random field with $supp(\mu)=X$ is a Gibbs state for a shift-invariant nearest neighbor interaction.}
\end{enumerate}
\end{thm}

The second statement in the theorem above is not a part of the original formulation, but does follow since there is an explicit algorithm to produce a nearest neighbor interaction which is invariant under all graph automorphisms for which the original Markov random field was invariant \cite{Chandgotia}. See also \cite{avernicev_MRF_gibbs_1972},\cite{sherman_MRF_gibbs_1973},\cite{spitzer}. It is in general false that a shift-invariant Markov random field whose (shift-invariant) specification is compatible with some nearest neighbor interaction must also be compatible with some shift-invariant nearest neighbor interaction (see Corollary \ref{cor:no_invaraint_interaction} below). In particular, for a general topological Markov field $X$ we have $\G_X^\sigma \subset \M_X^\sigma \cap \G_X$, but the inclusion may be strict.

An inspection of the original proof of the Hammersley-Clifford Theorem actually gives the following a priori stronger result:

\begin{thm} \textbf{(Hammersley-Clifford Theorem, strong version)}\label{thm: strong_hammersley_clifford}
Let $X$ be a topological Markov field with a safe symbol. Then:
\begin{enumerate}
\item{Any Markov cocycle on $X$ is a Gibbs cocycle given by a nearest neighbor interaction. In our notation this is expressed by: $\M_X=\G_X$.}
\item{Further if $X$ is a subshift, any shift-invariant Markov cocycle on $X$ is a Gibbs cocycle given by a shift-invariant nearest neighbor interaction. In our notation this is expressed by: $\M^{\sigma}_X = \G^{\sigma}_X$.}
\end{enumerate}
\end{thm}

It is easily verified that any topological Markov field $X$ which satisfies one of the conclusions of the ``strong version'' immediately satisfies the corresponding conclusion of the ``weak version''. We will demonstrate in the following section that the converse implication is false in general.


For $X$ with a safe symbol, the ``strong'' and ``weak'' versions are easily equivalent because of the following simple claim:
\begin{prop}\label{prop:safe_symbol_MRF_full_support}
Let $X$ be a topological Markov field with a safe symbol. Then any Markov random field $\mu$ adapted to $X$ has $supp(\mu)=X$.
\end{prop}
\begin{proof}
Let $\mu$ be a Markov random field adapted to $X$.
We need to show that for all finite $F \subset V$ and $a \in \mathcal{B}_F(X)$, $\mu([a]_F)>0$.
Denote $F \cup \partial F$ by $\tilde{F}$. Let $b \in \mathcal{B}_{\tilde{F}}(X)$ and $c \in \mathcal{B}_{\partial \tilde{F}}(X)$ satisfy $\mu([b\vee c]_{\tilde{F} \cup \partial \tilde{F}}) > 0$. In particular, $\mu([c]_{\partial \tilde{F}}) > 0$.
Let $\tilde{b} \in \mathcal{B}_{\tilde{F}}(X)$ be given by
$$\tilde{b}_n := \begin{cases}
b_n & n \in F\\
\star & n \in \partial F.
\end{cases}$$
Note that $\tilde{b} \vee c\in \B_{\tilde{F}\cup \partial{\tilde F}}(X)$ because $\star$ is a safe-symbol.
Again, by the safe symbol property it follows that $\tilde{a} \in \mathcal{B}_{\tilde{F}}(X)$ where:
$$\tilde{a}_n := \begin{cases}
a_n & n \in F\\
\star & n \in \partial F.
\end{cases}$$
Since $X$ is a topological Markov field, $\tilde{a} \vee c \in \mathcal{B}_{\tilde{F} \cup \partial \tilde{F}}(X)$.
Since $\mu$ is an adapted Markov random field it follows that $\mu([\tilde{a}]_{\tilde{F}} \mid [c]_{\partial \tilde{F}}) >0$, and since $\mu([c]_{\partial \tilde{F}}) > 0$ it follows that $\mu([\tilde{a} \vee c]_{\tilde{F} \cup \partial \tilde{F}}) >0 $; in particular we get that $\mu([a]_F)>0$.
\end{proof}

\textbf{Remark:} Proposition \ref{prop:safe_symbol_MRF_full_support} is a particular instance of the more general fact that all $\Delta_X$-nonsingular measures $\mu$ satisfy $supp(\mu)=X$, whenever $\Delta_X$ is a topologically minimal. The latter condition means that for any $x \in X$, the $\Delta_X$-orbit $\Delta_X(x):= \{ y \in X~:~ (x,y) \in \Delta_X\}$ is dense in $X$. 

\textbf{Remark:} When the underlying graph is $\mathbb{Z}^d$, any shift-invariant topological Markov field $X$ with a safe symbol is actually a nearest-neighbor shift of finite type. This follows using arguments similar to those appearing in the proof of Proposition \ref{prop:safe_symbol_MRF_full_support}.

\subsection{The Pivot Property}\label{section: pivot}
We shall now consider a weaker property than that of having a safe symbol.
Let $X$ be a topological Markov field.
If $x,y \in X$ only differ at a single $v \in V$, then the pair $(x,y)$ will be called \emph{a pivot move} in X.
A topological Markov field $X$ is said to have \emph{the pivot property} if for all $(x,y)\in \Delta_X$ such that $x\neq y$ there exists a finite sequence of points $x^{(1)}=x, x^{(2)},\ldots, x^{(k)}=y \in X$ such that each $(x^{(i)}, x^{(i+1)})$ is a pivot move. In this case we say $x^{(1)}=x, x^{(2)},\ldots, x^{(k)}=y$ is a chain of pivots from $x$ to $y$. Sometimes in the literature the pivot property is called ``local move connectedness''. 
Here are some examples of subshifts which have the pivot property:

\begin{enumerate}
\item Any topological Markov field with a trivial homoclinic relation.
\item Any topological Markov field with a safe symbol.
\item The ``3-colored chessboard'' (Proposition \ref{prop:X_r_pivot_property} below).
\item $r$-colorings of $\mathbb{Z}^d$ with $r\geq 2d+2$.\label{enumerate: pivot property r-colouring}
\item The space of graph homomorphisms from $\Z^d$ to a ``dismantlable graph'', as in \cite{brightwell2000gibbs}.
\end{enumerate}

The fact that $r$-colorings of $\mathbb{Z}^d$ with $r\geq 2d+2$ have the pivot property is a consequence of the following:
\begin{prop} \label{prop:coloring_pirvot}
Let $G=(V,E)$ be a graph with all vertices of degree at most $d$, and $n \ge d+2$. Let $C_n$ denote the proper vertex
colorings of $G$ with $n$ colors:
$$ C_n := \left\{ x \in \{1,\ldots,n\}^V ~:~ x_v \ne x_w \mbox{ whenever } (v,w) \in E\right\}.$$
$C_n$ is a topological Markov field that has the pivot property.
\end{prop}
\begin{proof}
Given $x,y \in \Delta_{C_n}$ we will describe a sequence of pivot moves from $x$ to $y$: Let $F := \{ v \in V~:~ x_v \ne y_v\}$.
Let $x^{(0)}:= x$. For every $i \in \{1,\ldots, n\}$, obtain a point $x^{(i)} \in C_n$ by pivot moves starting from $x^{(i-1)}$.  The points $x^{(i)}$ will have the property that $x^{(i)}_v=y_v$ unless $v \in F$ and $y_v >i$.
Roughly, we first replace all the occurrences of $i$ in $x^{(i-1)}|_F$ one by one, and then apply  more pivot moves until resulting configuration $x^{(i)}$ is equal to $y$ at all sites in $F$ where $y$ has a symbol less than or equal to $i$.
Precisely, the  pivot moves from  $x^{(i-1)}$ to  $x^{(i)}$ are applied in two steps as follows:
\begin{itemize}
\item Step $1$:
For every $v \in F$ such that $x^{(i-1)}_v= i$, choose $j \in \{1,\ldots,n\}$ such that $$j \not \in \left\{ x^{(i-1)}_w ~:~ w \in \partial \{v\} \cup \{v\} \right\}.$$ 
Such $j$ exists because by our assumption on $n$: $|\partial \{v\} \cup \{v\}| \le d+1 < n$. In particular $i \ne j$.
Make a pivot move by changing $x^{(i-1)}_v$ from $i$ to $j$.
After at most $|F|$ pivot moves we obtain the point $z$ such that $z_v \ne i$ for all $v \in F$ and $z_v = x^{(i-1)}_v$ unless $v \in F$ and $x^{(i-1)}_v=i$.
\item Step $2$: For every $v \in F$ such that $y_v =i$ apply a pivot move by changing $z_v$ to $i$:
\begin{equation*}
x^{\prime}_w:=\begin{cases}
z_w \text{ if } w\neq v\\
i \ \text{ if }w =v.
\end{cases}
\end{equation*}
To see that $x^{\prime} \in C_n$, we need to check that $z_w \ne i$ for any $w \in \partial \{v\}$. Indeed if $w \in \partial \{v\} \cap F$ then $z_w \ne i$ by Step $1$. If $w \in \partial \{v\} \cap F^c$ then $z_w = y_w$. Because $y_v =i$ it follows that $y_w \ne i$.
Now iterate with $z$ replaced by $x^{\prime}$. After at most $|F|$ pivot moves we obtain the point $x^{(i)}$. This configuration has the property that $x^{(i)}_v=y_v$ unless $v \in F$ and $y_v >i$.
\end{itemize}
This describes a finite sequence of pivot moves from $x=x^{(0)}$ to $x^{(n)}=y$.
\end{proof}

\begin{prop} \label{finite pivot dimension}
Let $X\subset \mathcal{A}^{\mathbb{Z}^d}$ be a shift-invariant topological Markov field with the pivot property. Then the dimension of $\M^{\sigma}_X$ is finite.
\end{prop}
\begin{proof}
Let $(x,y)\in \Delta_X$. Let $x=x^{(1)}, x^{(2)}, x^{(3)} ,\ldots ,x^{(k)}=y$ be a chain of pivots from $x$ to $y$. Then
\begin{equation}\label{cocyclepivot}
M(x,y)= \sum_{i=1}^{k-1}M(x^{(i)}, x^{(i+1)}).
\end{equation}
If $x^{(i)}, x^{(i+1)}$ differ only at $m_i \in \Z^d$ then $M(x^{(i)}, x^{(i+1)})=M(\sigma_{-m_i}x^{(i)}, \sigma_{-m_i}x^{(i+1)})$ depends only on ${\sigma_{-m_i}x^{(i)}|_{\{0\} \cup \partial \{0\}}}$ and ${\sigma_{-m_i}x^{(i+1)}|_{\{0\} \cup \partial \{0\}}}$. Therefore the dimension of the space of shift-invariant Markov cocycles is bounded by $|\B_{\{0\} \cup \partial \{0\}}(X)|^2$.
\end{proof}

\textbf{Remark:} More generally, the same proof shows that $\dim(\M_X) <\infty $ whenever there are a finite number of maps
$f_1,\ldots,f_k:X \to X$ and a finite $F \subset \mathbb{Z}^d$ so that $(f_i(x))_n=x_n$ whenever $n \not\in F$ and $\Delta_X$ is the orbit relation generated by
$$ \{\sigma_n^{-1} \circ f_i \circ \sigma_n ~:~ n \in \mathbb{Z}^d,~ i=1,\ldots,k \}.$$
This corresponds to subshifts with the ``generalized pivot property'': There is an allowed ``generalized pivot move'' from $x \in X$ to $y= f_i(x) \in X$ where $x_n$ and $y_n$ can differ only when $n$ is in a fixed finite set $F \subset \mathbb{Z}^d$. ``Domino tiles'' are an interesting and well known example for a subshift with the generalized pivot property \cite{signmatrixanddomino1}.

\section{ $\mathbb{Z}_r$-homomorphisms, $3$-colored chessboards and height functions}\label{sec:Zr_homomorphisms}

Recall that a \emph{graph-homomorphism} from the graph $G_1=(V_1,E_1)$ to the graph $G_2=(V_2,E_2)$ is a function $f:V_1 \to V_2$ from the vertex set of $G_1$ to the vertex set
of $G_2$ such that $f$ sends edges in $G_1$ to edges in $G_2$. Namely, if $(v,w) \in E_1$ then $(f(v),f(w)) \in E_2$.
We consider $\ZD$ as the vertex set of the standard Cayley graph, where an edge $(n,m)$ is present if and only if $\|n-m\|_1=1$. Also a subset $A \subset \ZD$ will denote the induced subgraph of $\ZD$ on $A$.

For the purposes of this paper, a \emph{height function} on $A \subset \Z^d$ is a graph homomorphism from $A$ to the standard Cayley graph of $\mathbb{Z}$.
We denote the set of height functions on $A \subset \ZD$ by $\Ht(A)$:

\begin{equation}
Ht^{(d)}(A) := \left\{ \hat{x} \in \Z^A ~:~ |\hat{x}_n - \hat{x}_m| =1 \mbox{ whenever } n,m \in A \mbox{ and } \| n-m\|_1 =1 \right\}.
\end{equation}

In particular, we denote
$$ Ht^{(d)} := Ht^{(d)}(\ZD)$$


We now introduce a certain family of $\mathbb{Z}^d$-shifts of finite type $X_r^{(d)}$, where $r,d \in \mathbb{N}$, and $r>1$:
Denote by $\mathbb{Z}_r = \mathbb{Z} / r\mathbb{Z} \cong \{0,\ldots,r-1\}$ the finite cyclic group of residues modulo $r$. Let $\phi_r: Ht^{(d)} \to (\mathbb{Z}_r)^{\mathbb{Z}^d}$ be defined by
$$(\phi_r(\hat{x}))_n := \hat{x}_n\!\!\! \mod r \mbox{ for all } n \in \mathbb{Z}^d.$$
The $\mathbb{Z}^d$-subshift $X_r^{(d)}$ is defined by:
$$ X^{(d)}_r := \phi_r(Ht^{(d)}).$$

For $r=2$ it is easily verified that $X_2^{(d)}$ consists precisely of two points $x^{even},x^{odd}$.
These are ``chessboard configurations'', given by $x^{even}_n = \|n\|_1 \mod 2$ and $x^{odd}_n = \|n\|_1+1 \mod 2 $.


In the following, to avoid cumbersome superscripts, we will fix some dimension $d \ge 2$, and denote $Ht := Ht^{(d)}$, $X_r:= X^{(d)}_r$ and $Ht(A):=Ht^{(d)}(A)$ for all $A\subset \Z^d$.

For $r \ne 1,4$, there is a direct and simple interpretation for the subshift $X_r$ as the set of graph homomorphisms from the standard Cayley graph of $\mathbb{Z}^d$ to the standard Cayley graph of $\mathbb{Z}_r$ (Proposition \ref{prop:X_r_def} below).
In the particular case $r=3$ the standard Cayley graph of $\mathbb{Z}_r$ is the complete graph on $3$ vertices, and so
$X_3$ is the set of proper vertex-colorings of the standard Cayley graph of $\mathbb{Z}^d$ with colors in $\{0,1,2\}$.
This relation has certainly been noticed and recorded in the literature. For instance, it is stated without proof in \cite{galvin_homomorphism_cube_2003}.
For the sake of completeness, we bring a self-contained proof in Proposition \ref{prop:X_r_def} and Lemma \ref{prop:height_well_defined} below. The proofs below essentially follow \cite{schmidt_cohomology_SFT_1995,schimdt_fund_cocycle_98}, where the corresponding results are obtained for the case $d=2$, $r=3$. Within the proof we also define a function $\DHt:X_r \times \ZD \to \mathbb{Z}$, which we use later on.

\begin{prop}\label{prop:X_r_def}
For any $d \ge 2$ and $r \in \mathbb{N} \setminus \{1,4\}$, $X_r$ is a nearest neighbor shift of finite type given by
$$ X_r = \{ x \in (\mathbb{Z}_r)^{\mathbb{Z}^d} ~:~ x_n - x_m = \pm 1 \mod r \mbox{, whenever } \|n-m\|_1=1 \}.$$
\end{prop}
\begin{proof}
When $r=2$, by our previous remark $X_2 = \{x^{odd},x^{even}\}$; the proposition is easily verified in this case.
From now on assume $r \in \mathbb{N} \setminus \{1,2,4\}$.
Temporarily, let us denote
$$ Y_r := \{ x \in (\mathbb{Z}_r)^{\mathbb{Z}^d} ~:~ x_n - x_m = \pm 1 \mod r \mbox{, whenever } \|n-m\|_1=1 \}.$$
We need to establish that $Y_r = X_r$.

For all $\hat{x} \in Ht$ and $m,n \in \mathbb{Z}^d$ with $\|n -m\|_1=1$, by definition of $Ht$, we have $|\hat{x}_n - \hat{x}_m| = 1$. Thus, $(\phi_r(\hat{x}))_n-(\phi_r(\hat{x}))_m = \pm 1 \mod r$, and so $\phi_r(\hat{x}) \in Y_r$. This establishes the inclusion $ X_r\subset Y_r $.

To complete the proof, given $x \in Y_r$ we will exhibit $\hat{x} \in Ht$ so that $\phi_r(\hat{x})=x$.
Choose $n_0 \in \mathbb{Z}^d$ and define $\hat{x}_{n_0} := x_{n_0}$.
For all $n \in \mathbb{Z}^d$, choose a path $n_0,n_1\ldots,n_k=n$ from $n_0$ to $n$ in the standard Cayley graph of $\mathbb{Z}^d$, meaning, $n_{i+1}-n_i \in \{\pm e_1,\ldots,\pm e_d\}$. Define
$$\hat{x}_n := \hat{x}_{n_0} + \sum_{i=1}^k [x_{n_i} - x_{n_{i-1}}],$$
where for $x \in Y_{r}$ and $m,n \in \ZD$ with $\|m-n\|_1 =1$,
\begin{equation}\label{eq:brk_def}
[x_{n}-x_{m}]:=\begin{cases}
1&\text{ if }x_{n}-x_{m}= 1\mod r \\
-1 &\text{ if }x_{n}-x_{m}= -1\mod r
\end{cases}.
\end{equation}

We claim that for all $x \in Y_r$ the value of $\hat{x}_n$
thus obtained is independent of the path chosen from $n_0$ to $n$.
Another way to express this is as follows:

For $x \in Y_r$ and $n \in \{\pm e_1,\ldots,\pm e_d\}$ let
\begin{equation}\label{eq:Dht_generators}
\DHt(x,n):= [x_n -x_0].
\end{equation}
Extend $\DHt$ to a map $\DHt:Y_r \times \ZD \to \mathbb{Z}$ as follows: For an arbitrary $n \in \ZD$ write $n=\sum_{j=1}^M s_j$, with $s_j \in \{\pm e_1,\ldots,\pm e_d\}$. Define:
\begin{equation}\label{eq:Ht_coycle}
\DHt(x,n) := \sum_{k=1}^M \DHt(\sigma_{n_{k-1}}x,s_k),
\end{equation}
where $n_k := \sum_{j=1}^k s_j$ and the expressions appearing in the sum on the right-hand side of \eqref{eq:Ht_coycle} are defined by \eqref{eq:Dht_generators}.
We will now verify that $\DHt(x,n)$ is well defined, which means it does not depend on the representation $n=\sum_{j=1}^M s_j$.
Specifically, we will check that for all $t_1,\ldots,t_N, s_1,\ldots,s_k \in \{\pm e_1,\ldots, \pm e_d\}$ satisfying $\sum_{i=1}^N t_i =\sum_{j=1}^k s_j$ and $x \in Y_r$
\begin{equation}\label{eq:indepedence_path_general}
\sum_{j=1}^k [x_{n_j}-x_{n_{j-1}}] = \sum_{j=1}^N [x_{m_j}-x_{m_{j-1}}],
\end{equation}
where $n_j = \sum_{i=1}^j s_i$ and $m_j = \sum_{i=1}^j t_i$.
From \eqref{eq:brk_def} it follows that $[x_n -x_m] = -[x_m - x_n]$ whenever $\|m-n\|_1 =1$. To check \eqref{eq:indepedence_path_general}, it thus suffices to verify that for all $i,j \in \{1,\ldots,d\}$ and $x \in Y_r$:

\begin{equation}\label{eq:4_cycle_ind}
[x_{e_j}-x_{0}] + [x_{e_j+e_i} - x_{e_j}] = [x_{e_i}-x_{0}] + [x_{e_i+e_j} - x_{e_i}]
\end{equation}
From the definition \eqref{eq:brk_def}, the equality in \eqref{eq:4_cycle_ind} holds modulo $r$. Also, $$[x_{e_j}-x_0],[x_{e_j+e_i} - x_{e_j}],[x_{e_i}-x_0],[x_{e_j+e_i} - x_{e_i}] \in \{\pm 1\}.$$
Thus \eqref{eq:4_cycle_ind} is a consequence of the following simple exercise:

For all $A_1,A_2,A_3,A_4 \in \{\pm 1\}$ satisfying $$A_1+A_2+A_3+A_4= 0 \mod r, \mbox{ where } r \in \mathbb{N} \setminus \{1,2,4\},$$ we have
$$A_1+A_2+A_3+A_4= 0.$$


It now follows that for all $x \in Y_r$ and $n \in \ZD$
\begin{equation}\label{eq:DHT_diff_lift}
\DHt(x,n)= \hat{x}_n - \hat{x}_0.
\end{equation}


In particular it follows that for all $n,m \in \mathbb{Z}^d$ with $\|n-m\|_1=1$, $\hat{x}_m -\hat{x}_n \in \{\pm 1\}$.
So indeed $\hat{x} \in Ht$. It is straightforward to check that $\phi_r(\hat{x})=x$.

\end{proof}

\textbf{Remark:} From the proof above we see that the map $\DHt:X_r \times \ZD \to \mathbb{Z}$ satisfies the following relation:
\begin{equation}\label{eq:DHt_coycle_relation}
\DHt(x,n+m)=\DHt(x,n)+ \DHt(\sigma_n(x),m).
\end{equation}
This means that $\DHt$ is a \emph{cocycle} for the shift-action on $X_r$. See \cite{schmidt_cohomology_SFT_1995,schimdt_fund_cocycle_98} for more on cocycles for $X_r$ and other subshifts.

On the first reading of the following sections, we advise the reader to keep in mind the case $r=3$ and $d=2$, in which $X^{(d)}_r$ is the ``$2$-dimensional $3$-colored chessboard''.


\textbf{Remark:} For the ``exceptional'' case $r=4$, $X_{4}$ is still a shift of finite type. This can be directly deduced from the following presentation:
\begin{eqnarray*}
X^{(d)}_4&=&\{x\in (\Z_r)^{\Z^d}~:~ x_n- x_m= \pm 1 \mod 4,\text{ whenever } ||n-m||_1=1\\
&&\text{ and } [x_{p}- x_{p+e_i}]+[x_{p+e_i}- x_{p+e_i+e_j}]=[x_{p}- x_{p+e_j}]+
[x_{p+e_j}-x_{p+e_i+e_j}]\\&&\text{ for all } 1\leq i, j\leq d \text{ and }p \in \Z^d\}.
\end{eqnarray*}
However $X^{(d)}_4$ is not a topological Markov field, as we now explain. For simplicity assume $d=2$. Let $x,y\in X^{(2)}_4$ be the periodic points satisfying
$$y_n = \sum_{i=1}^2 n_i \mod 4 \mbox{ and }x_n = 2-\sum_{i=1}^2 n_i \mod 4.$$ 
That is:
\begin{eqnarray}
y=
\begin{smallmatrix}
\ldots & \vdots & \vdots & \vdots & \ldots\\
\ldots& 1 & 2& 3 & \ldots\\
\ldots & 0 &1&2 & \ldots\\
\ldots & 3 & \bar{0}&1& \ldots\\
\ldots & 2 & 3&0& \ldots\\
\ldots & \vdots & \vdots & \vdots & \ldots
\end{smallmatrix}
\text{ and }
x=
\begin{smallmatrix}
\ldots & \vdots &\vdots & \vdots & \ldots\\
\ldots&1& 0& 3& \ldots\\
\ldots& 2& 1&0& \ldots\\
\ldots& 3 & \bar{2}&1& \ldots\\
\ldots& 0 & 3&2& \ldots\\
\ldots & \vdots & \vdots & \vdots & \ldots
\end{smallmatrix}
\end{eqnarray}
Observe that $x_0=2$ and $y_0=0$ and $x|_{\partial \{0\}}= y|_{\partial \{0\}}$. However the configuration $z$ given by
$$z_i=\begin{cases}
x_0\text{ if }i=0\\
y_i\text{ otherwise}
\end{cases}$$
is not an element of $X^{(2)}_4$ because
$$[z_{-e_2}-z_{-e_2+e_1}]+[z_{-e_2+e_1}-z_{e_1}]=-2$$
while
$$[z_{-e_2} -z_{0}]+[z_{0}-z_{e_1}]=2.$$
The same construction works for any $d\ge 2$.

\begin{lemma}\label{prop:height_well_defined}
Fix any $d \ge 2$ and $r \in \mathbb{N} \setminus \{1,2,4\}$.
For $x \in X_r$ any two pre-images under $\phi_r$ differ by a constant integer multiple of $r$, that is, if
$\hat{x},\hat{y} \in Ht$ satisfy $\phi_r(\hat{x})=\phi_r(\hat{y})$ then
there exists $M \in \mathbb{Z}$ so that $\hat{x}_n-\hat{y}_n=rM$ for all $n\ \in \mathbb{Z}^d$.
\end{lemma}

\begin{proof}
Let $x \in X_r$, $\hat{x},\hat{y} \in Ht$ satisfy $\phi_r(\hat{x})=\phi_r(\hat{y})=x$.
We have
$$\hat{x}_0 \equiv \hat{y}_0 \equiv x_0 \mod r.$$
Thus there exists $M \in \mathbb{Z}$ so that $\hat{x}_0- \hat{y}_0 =rM$.
By \eqref{eq:DHT_diff_lift} it follows that for all $n\in \ZD$
$$\hat{x}_n -\hat{x}_0 = \hat{y}_n-\hat{y}_0 = \DHt(x,n).$$
It follows that for all $n \in \ZD$
$$\hat{x}_n- \hat{y}_n = \hat{x}_0- \hat{y}_0 = rM.$$
\end{proof}


\begin{lemma}\label{lem:homoclinic_heights}
Fix any $(x,y) \in \Delta_{X_r}$ and a finite $F \subset \ZD$ so that $x_n=y_n$ for all $n \in \ZD \setminus F$.
\begin{enumerate}
\item{There exists a finite set $\tilde F$ so that for all $\hat{x},\hat{y} \in Ht$ such
that $\phi_r(\hat{x})=x$ and $\phi_r(\hat{y})=y$,
there exists $M \in \mathbb{Z}$ so that $\hat{x}_n - \hat{y}_n= rM$ for all $n \in \ZD \setminus \tilde F$.}
\item{We can choose $\hat{x}, \hat{y} \in Ht$ so that $M=0$, that is, $(\hat{x},\hat{y}) \in \Delta_{Ht}$.}
\item{If $(\hat{x}',\hat{y}'),(\hat{x},\hat{y}) \in \Delta_{Ht}$ and $\phi_r(\hat{x}')=\phi_r(\hat{x})=x$,
$\phi_r(\hat{y}')=\phi_r(\hat{y})=y$ then there exists $M \in \Z$ so that for all $n \in \ZD$, $\hat{y}'_n=\hat{y}_n+rM$ and $\hat{x}'_n=\hat{x}_n+rM$. }
\item{\label{item:homo4}For any $(\hat{x},\hat{y}) \in \Delta_{Ht}$ and all $n \in \Z^d$, $(\hat{x}_n - \hat{y}_n) \in 2\mathbb{Z}$.}
\item{\label{item:homo5} If $x,y \in X_r$ satisfy $x_n=y_n$ for all $n \in \Z^d \setminus \{n_0\}$ and $x_{n_0}\ne y_{n_0}$ then there exist $\hat{x}, \hat{y} \in Ht$ such that $\phi_r(\hat{x})=x, \phi_r(\hat{y})=y$ and
\begin{equation*}
|\hat{x}_n-\hat{y}_n|=\begin{cases}2\text{ if } n=n_0\\ 0 \text{ otherwise .}\end{cases}
\end{equation*}}
\end{enumerate}
\end{lemma}
\begin{proof}

\begin{enumerate}
\item
Choose $\hat{x} \in \phi_r^{-1}(x)$ and $\hat{y} \in \phi_r^{-1}(y)$.
Since the set $F$ is finite, there is an infinite connected component $\tilde A \subset \ZD \setminus F $ in the standard Cayley graph of $\ZD$ so that $\tilde F:= \ZD \setminus \tilde A$ is finite.
Fix some $n_0 \in \tilde{A}$. For any $n \in \tilde{A}$ choose a path $n_0,n_1\ldots,n_k=n\in \tilde A$ from $n_0$ to $n$ in the standard Cayley graph of $\mathbb{Z}^d$, that is, $n_{j+1}-n_{j} \in \{\pm e_1,\ldots,\pm e_d\}$ for all $0\leq j \leq k-1$ and $x_{n_j}= y_{n_j}$ for all $0\leq j \leq k$.
Using \eqref{eq:Dht_generators} and \eqref{eq:Ht_coycle} we conclude that
$$\DHt(\sigma_{n_0}(x),n- n_0)=\DHt(\sigma_{n_0}(y),n-n_0).$$
It now follows using \eqref{eq:DHT_diff_lift} that
$$\hat{x}_n - \hat{x}_{n_0} = \hat{y}_n -\hat{y}_{n_0}.$$
Because $x_{n_0}=y_{n_0}$, $\hat{x}_{n_0} - \hat{y}_{n_0} \in r\Z$, and so there exists $M \in \Z$ so that $\hat{x}_n -\hat{y}_n = rM$ for all $n \in \tilde{A}= \ZD \setminus \tilde F$.
\item
Define $\hat{z}$ by
$$\hat{z}_n:= \hat{y}_n - (\hat{y}_{n_0} -\hat{x}_{n_0}) \mbox{ for all } n\in \ZD.$$
Obviously $\hat{z} \in Ht$.
Since $x_{n_0}=y_{n_0}$ it follows that $\hat{y}_{n_0} -\hat{x}_{n_0} \in r\Z$. Thus $\phi_r(\hat{z})=\phi_r(\hat{y})=y$.
Also $\hat{z}_n=\hat{y}_n$ for all $n \in \tilde{A}$. Thus, $(\hat{x},\hat{z}) \in \Delta_{Ht}$, proving the second assertion.
\item
By Lemma \ref{prop:height_well_defined}, for any choice of $\hat{x}' \in \phi_r^{-1}(x)$ there exists $M_1 \in \Z$ so that $\hat{x}'_n = \hat{x}_n + rM_1$ for all $n \in \Z^d$.
Similarly, for any choice of $\hat{y}' \in \phi_r^{-1}(y)$ there exists $M_2 \in \Z$ so that $\hat{y}'_n = \hat{y}_n + rM_2$ for all $n \in \Z^d$.
But if $(\hat{x}, \hat{y}), (\hat{x}',\hat{y}') \in \Delta_{Ht}$ it follows that $M_1=M_2$. This proves the third assertion.

\item From \eqref{eq:Dht_generators} it follows that $\DHt(x,\pm e_j)= 1 \mod 2$ for all $x \in X_r$ and $j\in \{1,\ldots,d\}$. Thus by \eqref{eq:Ht_coycle} the parity of $\DHt(x,n)$ is equal to the parity of $\|n\|_1$ and does not depend on $x$. Therefore if $\hat{x}_{n_0} = \hat{y}_{n_0}$ for some $n_0 \in \ZD$, from \eqref{eq:DHT_diff_lift} it follows that
$\hat{x}_n -\hat{y}_n \in 2\Z$ for all $n \in \ZD$.

\item Suppose that $x,y \in X_r$ satisfy $x_n=y_n$ for all $n \in \Z^d \setminus \{n_0\}$ and $x_{n_0}\ne y_{n_0}$.
Choose $(\hat{x},\hat{y}) \in \Delta_{Ht}$ so that $\phi_r(\hat{x})=x$ and $\phi_r(\hat{y})=y$.
By the argument above $\hat{x}_n=\hat{y}_n$ for all $n \ne n_0$.
Let $m := n_0 + e_1$. Since $x_m=y_m$ and $x_{n_0} \ne y_{n_0}$ it follows that
$[x_{n_0}-x_m] \ne [y_{n_0}-y_m]$, thus
$$|[x_{n_0}-x_m] - [y_{n_0}-y_m]| =2.$$
Using \eqref{eq:Dht_generators} and \eqref{eq:DHT_diff_lift}, we conclude that
$|\hat{x}_{n_0}-\hat{y}_{n_0}|=2$.
\end{enumerate}
\end{proof}

It is a known and useful fact that the $3$-colored chessboard has the pivot property. This can be shown, for instance,
using height functions. Essentially the same argument shows that $X_r$ has the pivot property for all $r \in\mathbb{N}\setminus \{1,2,4\}$. We include a short proof below. Similar arguments appear in the proofs of certain claims in the subsequent sections.

\begin{prop}\label{prop:X_r_pivot_property}
For any $d \ge 2$ and $r \in \mathbb{N}\setminus \{1,2,4\}$ the subshift $X_r$ has the pivot property. In other words, given any $(x,y) \in \Delta_{X_r}$ there exist
$x=z^{(0)},z^{(1)},\ldots,z^{(N)}=y \in X_r$ such that for all $0\le k < N$, there is a unique $n_k \in \mathbb{Z}^d$ for which $z^{(k)}_{n_k} \ne z^{(k+1)}_{n_k}$.
\end{prop}

\begin{proof}
Fix $(x,y) \in \Delta_{X_r}$. By Lemma \ref{lem:homoclinic_heights}, we can choose $(\hat{x},\hat{y}) \in \Delta_{Ht}$ with $\phi_r(\hat{x})=x$ and $\phi_r(\hat{y})=y$. We will proceed by induction on $D = \sum_{n \in \mathbb{Z}^d}|\hat{x}_n-\hat{y}_n|$. Note that by Lemma \ref{lem:homoclinic_heights} $D$ is well defined, that is, the differences $(\hat{x}_n-\hat{y}_n)$ in the sum do not depend on the choice of $(\hat{x},\hat{y})$.
When $D=0$, then $x=y$ and the claim is trivial. Now, suppose $D >0$. Let
\begin{eqnarray*}
F_+& =& \{ n \in \mathbb{Z}^d~:~ (\hat{x}_n-\hat{y}_n) > 0 \}\\
F_- &=& \{ n \in \mathbb{Z}^d~:~ (\hat{x}_n-\hat{y}_n) < 0 \},
\end{eqnarray*} that is, $F_+ \subset \mathbb{Z}^d$ is the finite set of sites where $\hat{x}$ is strictly above $\hat{y}$ and $F_- \subset \mathbb{Z}^d$ is the finite set of sites where $\hat{y}$ is strictly above $\hat{x}$. Without loss of generality assume that $F_+$ is non-empty. Since $(\hat{x}_n-\hat{y}_n) \in 2\mathbb{Z}$, $\hat{x}_n-\hat{y}_n \ge 2$ for all $n \in F_+$. Consider $n_0\in F_+$ such that $\hat{x}_{n_0}=\max\{\hat{x}_n ~:~ n\in F_+\}$. It follows that $\hat{x}_{n_0} -\hat{x}_m =1$ for all $m$ neighboring $n_0$. We can thus define $\hat{z} \in Ht$ which is equal to $\hat{x}$ everywhere except at $n_0$, where $\hat{z}_{n_0}=\hat{x}_{n_0}-2$. Now set $z^{(1)}= \phi_r(\hat{z})$ and apply the induction hypothesis on $(z^{(1)},y)$.
\end{proof}

\section{ Markov cocycles on $X_r$}\label{sec:X_r_Markov_specifications}

Our goal in the current section is to describe the space of shift-invariant Markov cocycles on $X_r$, when $r \in \mathbb{N} \setminus \{1,2,4\}$, and the subspace of Gibbs cocycles for shift-invariant nearest neighbor interactions.

In the following we assume $r \in \mathbb{N} \setminus \{1,2,4\}$ and $d >1$, unless explicitly stated otherwise.
\begin{lemma}\label{lem:height_diff_markovian}
Let $F\subset \Z^d$ be a finite set and $x,y,z,w\in X_r$ such that $x|_{F^c}= y|_{F^c}$, $z|_{F^c}=w|_{F^c}$ and $x|_{F \cup \partial F} = z|_{F \cup \partial F}$, $y|_{F \cup \partial F} = w|_{F \cup \partial F}$. Consider $(\hat{x},\hat{y}), (\hat{z},\hat{w})\in \Delta_{Ht}$ such that they are mapped by $\phi_r$ to the pairs $(x,y), (z,w) \in \Delta_{X_r}$ respectively. Then $\hat{x}_n-\hat{y}_n= \hat{z}_n -\hat{w}_n$ for all $n \in \mathbb{Z}^d$.
\end{lemma}
\begin{proof}
Let $F_0 \subset \mathbb{Z}^d$ denote the infinite connected component of $F^c$. For $n\in F_0$, we clearly have $\hat{x}_n-\hat{y}_n=\hat{z}_n-\hat{w}_n=0$. We can now prove by induction on the distance from $n \in \mathbb{Z}^d$ to $F_0$ that
$\hat{x}_n-\hat{y}_n=\hat{z}_n-\hat{w}_n$. Given $n \in \mathbb{Z}^d \setminus F_0$, find a neighbor $m$ of $n$ which is closer to $F_0$. By the induction hypothesis, $\hat{x}_m-\hat{y}_m=\hat{z}_m-\hat{w}_m$.

If either $n \in F$ or $m \in F$, then both $m$ and $n$ are in $F \cup \partial F$ and so $x_n-x_m=z_n-z_m$ and $y_n-y_m=w_n-w_m$. \eqref{eq:brk_def} and \eqref{eq:DHT_diff_lift} imply $\hat{x}_n-\hat{x}_m=\hat{z}_n-\hat{z}_m$ and $\hat{y}_n-\hat{y}_m=\hat{w}_n-\hat{w}_m$.
Subtracting the equations and applying the induction hypothesis, we conclude in this case that $\hat{x}_n-\hat{y}_n=\hat{z}_n-\hat{w}_n$.

Otherwise, $n,m \in F^c$ and so $x_n-x_m=y_n-y_m$ and $z_n-z_m=w_n-w_m$, and again by \eqref{eq:brk_def} and \eqref{eq:DHT_diff_lift} we conclude that $\hat{x}_n-\hat{y}_n=\hat{z}_n-\hat{w}_n$.
\end{proof}

For $i \in \mathbb{Z}_r$ and integers $a,b$ with $a -b \in 2\mathbb{Z}$, let
$$N_i(a,b) := \begin{cases}
|\{ m \in (2\mathbb{Z}+a)\cap(r\mathbb{Z}+i)~:~ m \in [a,b)\}| & \mbox{ if } a\le b\\
-N_i(b,a) & \mbox{otherwise}
\end{cases}
$$
Here $N_i$ is the ``net'' number of crossings from $(i + r\mathbb{Z})$ to $(i+2+ r\mathbb{Z})$ in a path going from $a$ to $b$ in steps of magnitude $2$. Note that
\begin{equation}\label{crossingcount1}
N_{i}(a,b)=N_i(a+rn,b+rn)
\end{equation}
for all $a, b, n \in \Z$ and
\begin{equation}\label{crossingcount2}
N_i(a, b)= N_i(a+c, b+c)
\end{equation}
for all $a, b,c \in \Z$ such that $a-b\in r\Z\cap2\Z$ .

\begin{prop}\label{prop:X_r_markov_cocycles}
For any $r \in \mathbb{N} \setminus \{1,2,4\}$, the space $\mathcal{M}^{\sigma}_{X_r}$ of shift-invariant Markov cocycles on $X_r$ has a linear basis
$$\{ M_0, M_1,\ldots,M_{r-1}\},$$
where the cocycle $M_i$ is given by
\begin{equation}\label{eq:Markov_cocycles_basis}
M_i(x,y) := \sum_{n \in \mathbb{Z}^d} N_i(\hat{x}_n,\hat{y}_n);
\end{equation}
$(\hat{x},\hat{y}) \in \Delta_{Ht}$ is any pair which is mapped to $(x,y)$ by $\phi_r$. In particular $\dim( \mathcal M^\sigma_{X_r})=r$.

\end{prop}

\begin{proof}
By Lemma \ref{lem:homoclinic_heights} different choices $(\hat{x},\hat{y}) \in \Delta_{Ht}$ which map to $(x,y)$ via $\phi_r$ differ by a fixed translation in $r\Z$. Thus by \eqref{crossingcount1} the values $M_i(x,y)$ are independent of the choice of the corresponding height functions $(\hat{x},\hat{y}) \in \Delta_{Ht}$ and hence are well defined.

We will show that for $i=0,\ldots,r-1$, $M_i$ is indeed a Markov cocycle. Since $N_i(a,c)=N_i(a,b)+N_i(b,c)$ whenever $a \equiv b \equiv c \mod 2$, it follows that $M_i(x,z)=M_i(x,y)+M_i(y,z)$ whenever $x,y,z \in X_r$ are homoclinic. Thus $M_i$ is a $\Delta_{X_r}$-cocycle. Clearly, $M_i$ is shift-invariant.

We now check that $M_i$ satisfies the Markov property. This is equivalent to showing that $M_i(x,y)=M_i(z,w)$
whenever $x,y,z,w \in X_r$ satisfy the assumption in Lemma \ref{lem:height_diff_markovian}. In this case by Lemma \ref{lem:height_diff_markovian}, $\hat{x}_n-\hat{y}_n= \hat{z}_n -\hat{w}_n$ for all $n \in \mathbb{Z}^d$. Also note that for any $n \in \mathbb{Z}^d$,
either $x_n = z_n$ and $y_n= w_n$ in which case $\hat{x}_n - \hat{z}_n=\hat{y}_n-\hat{w}_n \in r\mathbb{Z}$ or $x_n=y_n$ and $z_n=w_n$, in which case by Lemma \ref{lem:homoclinic_heights}
$\hat{x}_n - \hat{y}_n =\hat{z}_n - \hat{w}_n \in r\mathbb{Z}\cap 2\mathbb{Z}$. By \eqref{crossingcount1} and \eqref{crossingcount2} in either case $N_i(\hat{x}_n,\hat{y}_n)= N_i(\hat{z}_n,\hat{w}_n)$ and summing over the $n$'s, we get $M_i(x,y)=M_i(z,w)$ as required.

To complete the proof we need to show that all shift-invariant Markov cocycles on $X_r$ can be uniquely written as a linear combination of $M_0,\ldots,M_{r-1}$.
For $i \in \{0,\ldots, r-1\}$ let $(x^{(i)},y^{(i)}) \in \Delta_{X_r}$ such that $x^{(i)}_0=i$, $y^{(i)}_0=i+2 \mod r$ and $x^{(i)}_n = y^{(i)}_n$ for all $n \in \mathbb{Z}^d \setminus \{0\}$.
Given a shift-invariant Markov cocycle $M$, let
\begin{equation}\label{eq:alpha_i_def}
\alpha_i := M(x^{(i)},y^{(i)})
\end{equation}
We claim that for all $(x,y) \in \Delta_{X_r}$:
\begin{equation}\label{eq:M_lin_comb}
M(x,y) = \sum_{i=0}^{r-1} \alpha_i \cdot M_i(x,y).
\end{equation}
Since $X_r$ has the pivot property (Proposition \ref{prop:X_r_pivot_property}), by \eqref{cocyclepivot} it is sufficient to show that \eqref{eq:M_lin_comb} holds for all $(x,y) \in \Delta_{X_r}$ which differ only at a single site.
By shift-invariance of $M$ and the $M_i$'s it is further enough to show this for $(x,y)$ which differ only at the origin $0$. In this case, we note that $(x,y)$ coincide with either $(x^{(i)},y^{(i)})$ or $(y^{(i)},x^{(i)})$ on the sites $\{0\}\cup\partial \{0\}$ for some $i$. Without loss of generality assume that $(x,y)$ coincide with $(x^{(i_0)},y^{(i_0)})$ on the sites $\{0\}\cup\partial \{0\}$. Since $M$ and the $M_i$'s are Markov cocycles we have $M(x,y)=M(x^{(i_0)},y^{(i_0)})=\alpha_{i_0}$ and $$\sum_{j=0}^{r-1} \alpha_j \cdot M_j(x,y)= \sum_{j=0}^{r-1}\alpha_j M_j(x^{(i_0)},y^{(i_0)})=\sum_{j=0}^{r-1}\alpha_j\delta_{i_0,j}=\alpha_{i_0}.$$
\end{proof}

\textbf{Remark: } Without the assumption of shift-invariance, a similar argument shows that any Markov cocycles on $X_r$ is of the following form:
$$M(x,y)= \sum_{i=0}^{r-1}\sum_{n \in \mathbb{Z}^d} a_{i,n} N_i(\hat{x}_n,\hat{y}_n) \mbox{ with } a_{i,n} \in \mathbb{R} \mbox{ for all } n \in \mathbb{Z}^d,~ 0\le i \le r-1.$$

We now describe the space $\mathcal{G}^\sigma_{X_r}$ of Gibbs cocycles corresponding to shift-invariant nearest neighbor interactions for $X_r$.

\begin{prop}\label{prop:Gibbs_cocycles_X_r}
A shift-invariant Markov cocycle on $X_r$ is a Gibbs cocycle corresponding to a shift-invariant nearest neighbor interaction if and only if it is of the form
$M=\sum_{i=0}^{r-1} \alpha_i M_i$, with $\sum_{i=0}^{r-1} \alpha_i = 0$ and $M_i$'s as in Proposition \ref{prop:X_r_markov_cocycles} and $\alpha_0,\ldots,\alpha_{r-1}$ given by \eqref{eq:alpha_i_def}. In other words,
$$\mathcal{G}^\sigma_{X_r} = \left\{\sum_{i=0}^{r-1} \alpha_i M_i~:~ \sum_{i=0}^{r-1} \alpha_i = 0\right\}.$$
In particular, $\dim( \mathcal{G}^\sigma_{X_r})=r-1$.
\end{prop}

\begin{proof}
Let $M$ be a Gibbs cocycle given by a shift-invariant nearest neighbor interaction $\phi$.
Choose $(x^{(i)},y^{(i)}) \in \Delta_{X_r}$ as in the proof of Proposition \ref{prop:X_r_markov_cocycles} so that $\alpha_i= M(x^{(i)},y^{(i)})$.
Expanding the Gibbs cocycle we have:
\begin{eqnarray*}M(x^{(i)},y^{(i)})&=&\phi([{i+2}]_0)-\phi([{i}]_0) \\&+& \sum_{j=1}^d \left( \phi([{i+2},{i+1}]_j) -(\phi([{i},{i+1}]_j)+ \phi([{i+1},{i+2}]_j) - \phi([{i+1},{i}]_j)\right).\end{eqnarray*}
Summing these equations over $i$ we get:
$$\sum_{i=0}^{r-1}{\alpha_i}=\sum_{i=0}^{r-1} M(x^{(i)},y^{(i)}) = 0$$
Conversely, for any values $\alpha_i=M(x^{(i)},y^{(i)})$ such that $\sum_{i=0}^{r-1} \alpha_i = 0$ it is easy to see that there is a corresponding nearest neighbor shift-invariant interaction $\phi$: For instance, set $\phi([{i},{i+1}]_1)=-\sum_{k=i}^{r-1}\alpha_k$,
$\phi([{i},{i+1}]_j)=0$ for $j=2,\ldots,d$ and $\phi([{i+1},{i}]_j)=\phi([{i}]_0)=0$ for $j=1,\ldots,d$ and $i=0,\ldots,{r-1}$.
\end{proof}

Let $\hat{M}: \Delta_{X_r} \to \mathbb{R}$ be the Markov cocycle given by
\begin{equation}\label{eq:hat_M}
\hat{M}(x,y) := \sum_{n \in \mathbb{Z}^d}\hat{y}_n - \hat{x}_n,
\end{equation}
where $(\hat{x},\hat{y})\in \Delta_{Ht}$ satisfy $\phi_r(\hat{x})=x$ and $\phi_r(\hat{y})=y$.
By the following we observe that $\hat{M}(x,y)= 2\sum_{i=0}^{r-1}M_i(x,y)$ for all $(x,y)\in \Delta_{X_r}$:

As in the proof of Proposition \ref{prop:X_r_markov_cocycles} it is sufficient to verify this for $(x,y)=(x^{(i_0)}, y^{(i_0)})$ where $0\leq i_0 \leq r-1$. In that case
$$\hat{M}(x^{(i_0)}, y^{(i_0)})=\hat{y}^{(i_0)}_0-\hat{x}^{(i_0)}_0=2$$
and
$$2\sum_{i=0}^{r-1}M_i(x^{(i_0)}, y^{(i_0)})= 2 \sum_{i=0}^{r-1} \delta_{i_0, i} =2.$$

\begin{corollary}\label{cor:X_r_cocycle_decomp}

Any shift-invariant Markov cocycle $M$ on $X_r$ can be uniquely written as
\begin{equation*}
M=M_0+\alpha \hat{M}
\end{equation*}
where $M_0$ is some Gibbs cocycle, $\alpha\in \R$ and $\hat{M}$ is given by \eqref{eq:hat_M}.
\end{corollary}
Thus, the conclusion of the second part of the strong version of the Hammersley-Clifford Theorem
regarding shift-invariant Markov cocycles fails for $X_r$. Our next proposition asserts that the conclusion of the first part of the strong version of the Hammersley-Clifford Theorem still holds for $X_r$. This immediately implies the conclusion of the first part of the weak version of the Hammersley-Clifford Theorem of $X_r$.

\begin{prop}\label{prop:X_r_Markov_is_non_stationary_gibbs} ($\mathcal{M}_{X_r} = \mathcal{G}_{X_r}$)

Let $M:\Delta_{X_r} \to \mathbb{R}$ be a Markov cocycle. There exists a nearest neighbor interaction
$\phi$, which is not necessarily shift-invariant, so that $M= M_\phi$.

\end{prop}
\begin{proof}
Given $M \in \mathcal{M}_{X_r}$, we will define a compatible nearest neighbor interaction $\phi$ as follows:

The interaction $\phi$ will assign $0$ to any single site configuration.
For $n=(n_1\ldots,n_d) \in \Z^d$, and $1 \le j \le d$, let $\phi_{n,j}(a,b)$ denote the weight the interaction $\phi$ assigns to the configuration $(a,b)$ on the edge $(n,n+e_j)$.
Set $\phi_{n,j}(a,b)=0$ whenever $j \in \{2,\ldots,d\}$. 
For $n=(n_1,\ldots,n_d) \in \Z^d$ and $i \in \Z_r$ define
\begin{eqnarray*}\phi_{n,1}(i,i+1) :=& \begin{cases}
0 & n_1\le 0\\
M(\sigma_n y^{(i)},\sigma_n x^{(i)})+\phi_{n-e_1,1}(i+1,i+2) & n_1 >0
\end{cases}\\
\phi_{n,1}(i+1,i) :=& \begin{cases}
0 & n_1 \ge 0\\
M(\sigma_{n+e_1} y^{(i)},\sigma_{n+e_1} x^{(i)})+\phi_{n+e_1,1}(i+2,i+1) & n_1 <0\\
\end{cases}\end{eqnarray*}
where, as in the proof of Proposition \ref{prop:X_r_markov_cocycles}, $(x^{(i)},y^{{i}}) \in \Delta_{X_r}$ are such that $x^{(i)}_0=i \mod r$, $y^{(i)}_0=i+2 \mod r$ and $x^{(i)}_n = y^{(i)}_n$ for all $n \in \mathbb{Z}^d \setminus \{0\}$.
This is well defined by induction on $|n_1|$. To see that $\phi$ defines a nearest neighbor interaction for $M$, since $X_r$ has the pivot property (Proposition \ref{prop:X_r_pivot_property}) it is sufficient to verify by \eqref{cocyclepivot}
\begin{equation}\label{eq:non_stationary_interaction}
M(y,x) = \sum_{n \in \Z^d}\sum_{j=1}^d\phi_{n,j}(x_{n},x_{n+e_j})-\phi_{n,j}(y_{n},y_{n+e_j}).
\end{equation}
for $(y,x)\in \Delta_{X_r}$ which differ at a single site $n^\p\in \Z^d$.

Then $(y,x)$ coincide with either $(\sigma_{n^\p}y^{(i)},\sigma_{n^\p}x^{(i)})$ or $(\sigma_{n^\p}x^{(i)},\sigma_{n^\p}y^{(i)})$ on the sites $\{n^\p\}\cup\partial \{n^\p\}$ for some $0\leq i\leq r-1$. Without loss of generality assume that $(y,x)$ coincide with $(\sigma_{n^\p}y^{(i_0)},\sigma_{n^\p}x^{(i_0)})$ on the sites $\{n^\p\}\cup\partial \{n^\p\}$ for some $0\leq i_0\leq r-1$. Since $M$ is a Markov cocycle
\begin{eqnarray*}
M(y,x)&=&M(\sigma_{n^\p}y^{(i_0)},\sigma_{n^\p}x^{(i_0)})\\
&=&\begin{cases}
\phi_{n^\p,1}(i_0, i_0+1)- \phi_{n^\p-e_1, 1}(i_0+1, i_0+2) \text{ if } n^\p_1>0\\
\phi_{n^\p- e_1,1}(i_0+1, i_0)- \phi_{n^\p, 1}(i_0+2, i_0+1) \text{ if } n^\p_1\leq 0
\end{cases}\\
&=&\sum_{n \in \Z^d}\sum_{j=1}^d\phi_{n,j}(x_{n},x_{n+e_j})-\phi_{n,j}(y_{n},y_{n+e_j}).
\end{eqnarray*}•
\end{proof}

Combining the above results we obtain:
\begin{corollary}\label{cor:no_invaraint_interaction}
There exists a shift-invariant Markov cocycle on $X_r$ which is given by a nearest neighbor interaction but not by a shift-invariant nearest neighbor interaction, that is,
$$ \mathcal{G}_{X_r}^\sigma \neq \mathcal{G}_{X_r} \cap \mathcal{M}_{X_r}^\sigma.$$
\end{corollary}

\section{Markov random fields on $X_r$ are Gibbs} \label{Section: MRF on 3cb gibbs}

Our main goal is to prove the following result:
\begin{thm}\label{thm:MRF_on_X_r_is_Gibbs}
Any shift-invariant Markov random field adapted to $X_r$ is a Gibbs state for some shift-invariant nearest neighbor interaction.
In particular any shift-invariant Markov random field $\mu$ with $\mathit{supp}(\mu)= X_r$ is a Gibbs state for some shift-invariant nearest neighbor interaction.
\end{thm}

Theorem \ref{thm:MRF_on_X_r_is_Gibbs} implies that the conclusion of the second part of the weak version of the Hammersley-Clifford Theorem holds for $X_r$, although the argument is very different from the safe-symbol case.

For a subshift $X$, a point $x\in X$ will be called \emph{frozen} if its homoclinic class is a singleton. This notion coincides with the notion of frozen coloring in \cite{brightwell2000gibbs}.
By Proposition \ref{prop:X_r_pivot_property}, $X_r$ has the pivot property so $x \in X_r$ is frozen if and only if for every $n \in \Z^d$, $x_j \neq x_k$ for some $j, k \in \partial\{n\}$, that is, any site is adjacent to at least two sites with distinct symbols. A subshift $X$ will be called \emph{frozen} if it consists of frozen points, equivalently $\Delta_X$ is the diagonal. A measure on a subshift $X$ will be called \emph{frozen} if its support consists of frozen points. Note that the collection of frozen points of a given topological Markov field $X$ is itself a topological Markov field.

We derive Theorem \ref{thm:MRF_on_X_r_is_Gibbs} as an immediate corollary of the following proposition:
\begin{prop}\label{prop:non_gibbs_frozen}
Let $\mu$ be a shift-invariant Markov random field adapted to $X_r$ with Radon-Nikodym cocycle equal to the restriction of $e^M$ to its support where $M \in \mathcal{M}^\sigma_{X_r} \setminus \mathcal{G}^\sigma_{X_r}$ is a Markov cocycle which is not given by a shift-invariant nearest neighbor interaction. Then $\mu$ is frozen.
\end{prop}
Note that any frozen probability measure is Gibbs with any nearest neighbor interaction because the homoclinic relation of the support of the measure is trivial. The intuition behind the proof of this proposition is the following: For a Markov cocycle $M= \sum_{i=1}^r \alpha_i M_i$ the condition $\sum_{i=1}^r\alpha_i>0$ indicates an inclination to raise the height function. However $\sigma$-invariance implies the existence of a well defined ``slope'' for the height function in all directions. Unless this slope is extremal, that is, maximal ($\pm \|n\|_1$) in some direction $n\in \Z^d\setminus\{0\}$, this will lead to a contradiction.

In preparation for the proof, we set up some auxiliary results.

\subsection{Real valued cocycles for measure-preserving $\Z^d$-actions}
We momentarily pause our discussion about Markov random fields on $X_r$ to discuss cocycles for measure-preserving $\mathbb{Z}^d$ actions.
Let $(X,\mathcal{F},\mu,T)$ be an ergodic measure-preserving $\mathbb{Z}^d$-action.
A measurable function $c:X\times \mathbb{Z}^d \to \mathbb{R}$ is called a \emph{$T$-cocycle} if it satisfies the following equation $\mu$-almost everywhere with respect to $x \in X$:
\begin{equation}\label{eq:T_cocycle}
c(x,n+m) =c(x,n)+c(T_{n}x,m) ~ \forall n,m \in \ZD .
\end{equation}

By \eqref{eq:DHt_coycle_relation} the function $\DHt:X_r \times \mathbb{Z}^d \to \mathbb{R}$ defined in \eqref{eq:Dht_generators} and \eqref{eq:Ht_coycle} is indeed a shift-cocycle with respect to the shift action on $X_r$.

We will use the following lemma:
\begin{lemma}\label{lem:coycle_ergodic}

Let $(X,\mathcal{F},\mu,T)$ be an ergodic measure-preserving $\mathbb{Z}^d$ action and $c:X\times \mathbb{Z}^d \to \mathbb{R}$ be a measurable cocycle such that
for all $n \in \mathbb{Z}^d$ the function $f_{c,n}(x):=c(x,n)$ is in $L^1(\mu)$, then
for all $n=(n_1, n_2,\ldots, n_d) \in \mathbb{Z}^d$
\begin{eqnarray*}\lim_{k \to \infty}\frac{c(x,k n)}{k}&=&\int c(x,n) d\mu(x)\\&=&\sum_{i=1}^d n_i\int c(x, e_i) d \mu(x).\end{eqnarray*}
The convergence holds almost everywhere with respect to $\mu$ and also in $L^1(\mu)$.
\end{lemma}

\begin{proof}
By the cocycle equation \eqref{eq:T_cocycle} for $\mu$-almost every $x \in X$, any $k \in \mathbb{N}$ and $n \in \ZD$ we have:
$$c(x,k\cdot n) = \sum_{i=0}^{k-1}c(T_{n}^{i}x,n).$$
The existence of almost everywhere and $L^1$ limit $\bar{f}(x):=\lim_{k \to \infty}\frac{c(x,k n)}{k}$ follows from the pointwise and $L^1$ ergodic theorems.
To complete the proof we need to show that the limit is constant almost everywhere. We do this by showing that the limit is $T$-invariant.
By the cocycle equation \eqref{eq:T_cocycle}, for almost every $x \in X$ and any $m,n \in \mathbb{Z}^d$ and $k \in \mathbb{N}$ we have:
\begin{eqnarray*}
c(x,k n) = & c(x,m+ kn -m)\\
= &c(x,m) + c(T_mx, kn -m)\\
=&c\left(x,m\right)
+c\left(T_{m}x,k n\right)
+c\left(T_{m+k n}x,-m\right).
\end{eqnarray*}
Thus,
$$| \bar{f}(x)-\bar{f}(T_m(x))|\leq\limsup_{k \to \infty}\frac{1}{k}\left(| c(x,m)| + |c(T_{m+k\cdot n}x,-m)|\right)\le \limsup_{k \to \infty}\frac{1}{k}|f_{c,m}(x)| +\frac{1}{k}|f_{c,-m}(T_n^kT_mx)| .$$
Since $f_{c,m}, f_{c,-m} \in L^1(\mu)$, $\limsup_{k \to \infty}\frac{1}{k}|f_{c,m}(x)|$ and $\limsup_{k \to \infty}\frac{1}{k}|f_{c,-m}(T_n^kT_mx)| $ are both equal to $0$ almost everywhere (the second term vanishes because $\lim_{k \to \infty}\frac{1}{k}g(S^k x) =0$ a.e for $g \in L^1$ and $S$ measure-preserving).
Therefore,
\begin{eqnarray*}
\lim_{k \to \infty}\frac{c(x,k\cdot n)}{k}&=&\int c(x,n) d\mu(x)\\&=&\sum_{i=1}^d n_i\int c(x, e_i) d \mu(x).\end{eqnarray*}
\end{proof}

\textbf{Remark:} In the specific case that $T$ is \emph{totally ergodic}, meaning that the individual action of each $T_{n}$ is ergodic for all $n \in \mathbb{Z}^d \setminus \{0\}$, the lemma above is completely obvious since
$\frac{c(x,k
\cdot n)}{k} = \frac{1}{k}\sum_{j=0}^{k-1}c(T^j_{n}x,n)$, which is an ergodic average. The point of Lemma \ref{lem:coycle_ergodic} is that ergodicity of the $\ZD$-action is sufficient for the limit to be constant.

The cocycle $\DHt:X_r \times \ZD \to \Z$ is not only measurable but also continuous. We can use this, along with compactness of $X_r$ and the unit ball in $\mathbb{R}^d$ to obtain uniformity of the convergence with respect to the ``direction'' on a set of full measure.

For convenience we extend the definition of $\DHt:X \times \ZD \to \mathbb{Z}\subset \R$ given by \eqref{eq:Dht_generators} and \eqref{eq:Ht_coycle} to a function $\DHt : X\times \R^d\longrightarrow \R$ as follows:
\begin{equation}
grad(x, w):= grad(x, \lfloor w \rfloor)
\end{equation}
where $w=(w_1, w_2, \ldots, w_d)\in \R^d$ and $\lfloor w \rfloor$ denotes $(\lfloor w_1 \rfloor, \lfloor w_2 \rfloor, \ldots, \lfloor w_d \rfloor)$.
\begin{lemma}\label{lem:height_slope_limit}
Let $\mu$ be an ergodic measure on $X_r$.
Then $\mu$-almost surely
$$\lim_{k \to \infty}\sup_{ \|w\|_1=k }\frac{1}{k}\left| \DHt(x,w) - \langle w, v\rangle \right|=0$$
where 
$$ v_j := \int \DHt(x,e_j) d\mu(x) ~,~ v:=(v_1,\ldots,v_d),$$
the supremum is over $\{ w \in \mathbb{R}^d ~:~ \|w\|_1 = k\}$,
and $\langle n, v\rangle = \sum_{i=1}^d n_i v_i$ is the standard inner product.
\end{lemma}

\begin{proof}
Let
$$E_\epsilon := \left\{x \in X_r ~:~ \limsup_{k \to \infty}\sup_{ \|w\|_1=k }\frac{1}{k}\left| \DHt(x,w) - \langle w, v\rangle \right| > \epsilon \right\}.$$

We will prove the lemma by showing that $\mu(E_\epsilon) =0$ for all $\epsilon >0$.

Fix $\epsilon >0$. Since $\mathbb{Q}^d$ is dense in $\mathbb{R}^d$, using compactness of the unit ball in $(\mathbb{R}^d,\|\cdot\|_1)$,
we can find a finite $F \subset \mathbb{Q}^d$ which $\frac{\epsilon}{8}$-covers the unit ball. By this we mean that for all $w \in \mathbb{R}^d$ such that $\|w\|_1=1$ there exists $u \in F$ so that $\|w-u\|_1 \le \frac{\epsilon}{8}$.
Because $F \subset \mathbb{Q}^d$ is finite, there exists $M \in \mathbb{N}$ so that $Mu\in \mathbb{Z}^d$ for all $u \in F$.
By Lemma \ref{lem:coycle_ergodic}, there exists a measurable set $X' \subset X_r$ with $\mu(X')=1$ so that for all $u \in F$, $x \in X'$
$$\lim_{k \to \infty} \frac{1}{Mk}\DHt(x,Mku)= \langle u, v\rangle.$$

To complete the proof we will prove that $X^\prime \subset E_\epsilon^c$.
Given $x \in X'$ we can find
an integer $J > 8M\epsilon^{-1}$ so that for all $k > \frac{J}{M}$ and all $u \in F$
$$\left|\frac{1}{Mk}\DHt(x,Mku)-\langle u,v\rangle \right| < \frac{\epsilon}{8}.$$
Note that for all $j>J$
$$\left|j-M\left\lfloor \frac{j}{M}\right\rfloor \right|\leq \frac{\epsilon}{8}j$$
Consider some $w \in \R^d$ such that $\|w\|_1=j >J$. We can find $u \in F$ so that
$\|\frac{w}{\|w\|_1}-u\|_1 \le \frac{\epsilon}{8}$.
Let $\tilde k := \lfloor\frac{j}{M}\rfloor$.
Then
\begin{equation}\label{eq:DHt_approx1}
\|w-\tilde k M u\|_1 \le \| w - j u\|_1 + |j- \tilde k M| \|u\|_1 \le \frac{\epsilon}{4} j.
\end{equation}
Now observe that $| \DHt(x,w)| \le \|w\|_1$ for all $x \in X_r$, $w \in \ZD$. From the cocycle property \eqref{eq:DHt_coycle_relation} it follows that for all $n,m \in \ZD$, $x \in X_r$:
$$ \DHt(x,n) = \DHt(x,m) + \DHt(\sigma_m x, n-m).$$
Therefore for all $w', u' \in \mathbb{R}^d$
$$| \DHt(x,w') - \DHt(x,u')| \le \|w'-u'\|_1+2d.$$

Applying the above inequality with $w'=w$ and $u'= \tilde kMu$ it follows using \eqref{eq:DHt_approx1}
$$\left| \DHt(x,w) - \DHt(x,\tilde k M u)\right| \le \frac{\epsilon}{4}j+2d.$$

Also, since $\|v\|_\infty \le 1$, it follows using \eqref{eq:DHt_approx1} that
$$ \left|\langle w,v \rangle - \langle\tilde k M u,v \rangle\right| < \frac{\epsilon}{4} j$$
which yields that for sufficiently large $j$,
$$\frac{1}{j}\left|\DHt(x,w)-\langle w,v\rangle \right| < \epsilon.$$
This proves that $X^\prime \subset E_{\epsilon}^c$.
\end{proof}

\subsection{Maximal height functions}
For $\hat{x} \in Ht$ and a finite $F \subset \mathbb{Z}^d$, let \begin{equation*}
Ht_{\hat{x},F}:=\{\hat{y} \in Ht \;|\; \hat{y}_n=\hat{x}_n \text{ if } n \not\in F \}.
\end{equation*}

Consider the partial ordering on $Ht_{\hat{x},F}$ given by $\hat{y} \geq \hat{z}$ if $\hat{y}_n \geq \hat{z}_n$ for all $n \in \Z^d$.
\begin{lemma}\label{lem:heightmaximum}
Let $\hat{x}\in Ht$ be given. Then $(Ht_{\hat{x},F}, \geq)$ has a maximum. If the maximum is attained by the height function $\hat{y}$ then for all $n \in F$:
\begin{equation*}
\hat{y}_{n}= \min\{\hat{x}_{k}+||n-k||_1~:~ k \in \partial F\}.
\end{equation*}
\end{lemma}

\begin{proof}
We will first prove that given height functions $\hat{y}, \hat{z} \in Ht_{\hat{x},F}$ the function $\hat{w}$ defined by
\begin{eqnarray*}
\hat{w}_i:= \max(\hat{y}_{i}, \hat{z}_i).
\end{eqnarray*}
is an element of $Ht_{\hat{x},F}$.

To see that $\hat{w}$ is a valid height function, we will show that $|\hat w_i - \hat w_j|=1$ for all two adjacent sites $i,j\in \Z^d$.
If $(\hat{w}_i,\hat{w}_j)=(\hat{y}_i,\hat{y}_j)$ or $(\hat{w}_i,\hat{w}_j)=(\hat{z}_i,\hat{z}_j)$ then $| \hat w_i - \hat w_j| =1$ because $\hat y, \hat z \in Ht$.
Otherwise, we can assume without loss of generality that $\hat{w}_i= \hat{y}_i > \hat{z}_i$ and $\hat{w}_j= \hat{z}_j > \hat{y}_j$.
By Lemma \ref{lem:homoclinic_heights}, because $(\hat{y}, \hat{z})\in \Delta_{Ht}$ we have
$$\hat{y}_i- \hat{z}_i, \hat{y}_j -\hat{z}_j\in 2\Z.$$
Thus $\hat{y}_i \ge \hat{z}_i +2$ and $\hat{z}_j \ge \hat{y}_j +2$.
Since $\hat{y}, \hat{z}\in Ht$ we have:
$$ \hat y_j + 1 \ge \hat y_i \ge \hat z_i +2 \ge \hat z_j +1 \ge \hat y_j +3,$$
a contradiction.

We conclude that $\hat w \in Ht$.
Also $\hat{w}_i= \max(\hat{y}_i,\hat{z}_i)= \hat{x}_{i}$ for all $i \in F^c$. Hence $\hat{w}\in Ht_{\hat x,F}$.

Since $Ht_{\hat x,F}$ is finite, it has a maximum.

Suppose the maximum is attained by a height function $\hat{y}$. Let $i \in F$, $k \in \partial F$ and $(i_1=i), i_2, i_3, \ldots,i_p, (i_{p+1}=k)$ be a shortest path between $i$ and $k$. Then
\begin{equation*}
\hat{y}_i= \sum_{t=1}^p \hat{y}_{i_t}-\hat{y}_{i_{t+1}} +\hat{y}_k=\sum_{t=1}^p \hat{y}_{i_t}-\hat{y}_{i_{t+1}} +\hat{x}_k.
\end{equation*}
Therefore $\hat{y}_i\leq ||i-k||_1+\hat{x}_k$ which proves that
\begin{equation}\label{onesideofineq}
\hat{y}_i\leq \min\{\hat{x}_k+||i-k||_1~:~ k \in \partial F\}.
\end{equation}
For proving the reverse inequality, note that if $\hat{y}$ has a local minimum at some $n \in F$ then the height at $n$ can be increased. Since $\hat{y}$ is the maximum height function, for each $n\in F$ at least one of the adjacent sites $m$ must satisfy $\hat{y}_n-\hat{y}_m=1$. Iterating this argument, for all $i \in F$, we can choose a path $j_1, j_2 ,j_3 ,\ldots, j_{p+1}$, with $j_1=i$, $j_2,\ldots,j_p \in F$, $j_{p+1} \in \partial F$ along which $\hat{y}$ is increasing: $\hat{y}_{j_t}- \hat{y}_{j_{t+1}}=1$ for all $t\in \{1, 2,\ldots, p \}$. Then
\begin{equation*}
\hat{y}_i=\sum_{t=1}^p \hat{y}_{j_t}-\hat{y}_{j_{t+1}} +\hat{y}_{j_{p+1}}\geq ||i-j_{p+1}||_1+\hat{y}_{j_{p+1}}.
\end{equation*}
Combining with the inequality $\eqref{onesideofineq}$, we get
\begin{equation*}
\hat{y}_i= \min\{\hat{x}_k+||i-k||_1~:~ k \in \partial F\}.
\end{equation*}
\end{proof}

Consider a shift-invariant Markov cocycle $M \in \mathcal{M}^{\sigma}_{X_r}$. Recall that by Corollary \ref{cor:X_r_cocycle_decomp} there exists $\alpha \in \mathbb{R}$ and a Gibbs cocycle $M_0 \in \mathcal{G}^{\sigma}_{X_r}$ compatible with a shift-invariant nearest neighbor interaction such that $M=M_0 + \alpha \hat{M}$. The following lemma is based on the idea that in the ``non-Gibbsian'' case $\alpha \ne 0$, whenever $\hat{y}$ is much bigger than $\hat{x}$, $M(x,y)$ is roughly $\alpha$ times the `volume' of the $(d+1)$-dimensional `shape' bounded between the graphs of $\hat{y}$ and $\hat{x}$ in $\Z^d\times \Z$.

For $N \in \mathbb{N}$, let
\begin{equation}\label{eq:D_N_def}
D_N:= \left\{n\in \Z^d~:~ \|n\|_1 \leq N\right\}
\end{equation}
be the ball of radius $N$ centered at the origin in the standard Cayley graph of $\ZD$. Also, denote:
\begin{equation}
S_N := \left\{ n \in\ZD ~:~ \|n\|_1=N \right\}
\end{equation}
Note that $S_N = \partial (D_N^c) = \partial D_{N-1}$.
\begin{lemma} \label{lem:bigmeasure}
Let $M=M_0 + \alpha \hat{M}$ be a shift-invariant Markov cocycle on $X^{(d)}_r$ where $M_0 \in \mathcal{G}^{\sigma}_{X_r^{(d)}}$ is a Gibbs cocycle compatible with a shift-invariant nearest neighbor interaction, $\hat{M}$ is the Markov cocycle given by \eqref{eq:hat_M} and $\alpha >0$. Then there exist a positive constant $c_1>0$ (depending only on $d$) and another positive constant $c_2>0$ (depending only on $d$ and $M_0$) such that for all $N \in \mathbb{N}$
\begin{equation*}
M(x,y)\geq c_1\alpha (\hat{y}_0-\hat{x}_0)^{d+1}- c_2\cdot N^{d}
\end{equation*}
for all $(\hat{x}, \hat{y}) \in \Delta_{Ht}$ satisfying $\hat x \le \hat y$, $x=\phi(\hat x)$, $y = \phi(\hat y)$ and $x|_{D_N^c}= y|_{D_N^c}$.

\end{lemma}
\begin{proof}
Let $M=M_0 + \alpha \hat{M}$, $(x,y) \in \Delta_{X_r}$ and $(\hat{x},\hat{y}) \in \Delta_{Ht}$ be as given in the lemma.
First we show that there exists a suitable constant $c_1>0$ (depending on $d$)
so that
\begin{equation}\label{eq:hat_M_estimate1}
\hat{M}(x,y) \ge c_1 (\hat{y}_0-\hat{x}_0)^{d+1}.
\end{equation}
Assume that $\hat{y}_0-\hat{x}_0>0$.
Denote:
\begin{equation}
K:= \frac{\hat{y}_0- \hat{x}_0}{2}
\end{equation}
Recall that $(\hat{x},\hat{y}) \in \Delta_{Ht}$, so by Lemma \ref{lem:homoclinic_heights}, $\hat{y}_n - \hat{x}_n \in 2\Z$ for all $n \in \ZD$. In particular, $K$ is an integer.
Since $\hat{y}_n- \hat{x}_n \ge 0$ we have:
$$ \hat M(x,y) \ge
\sum_{n \in D_N}(\hat{y}_n- \hat{x}_n)\geq \sum _{j=0}^{K}\sum_{n\in S_j}(\hat{y}_n- \hat{x}_n).
$$
Since $\hat{y}_n-\hat{x}_n\geq \hat{y}_0- \hat{x}_0-2\|n\|_1$:
$$\sum _{j=0}^{K}\sum_{n\in S_j}(\hat{y}_n- \hat{x}_n) \geq \sum_{j=0}^{K}|S_j|(\hat{y}_0- \hat{x}_0-2j).$$

Finally the estimates
$$|S_j | \geq |\{n\in S_j~:~ n \geq 0\}| ={{j+d-1}\choose{d-1}} \ge \frac{1}{d!}j^{d-1}$$
give
\begin{eqnarray*}
\hat M(x,y) \ge & \frac{1}{d!}\sum_{j=0}^K j^{d-1}(2K-2j) &\ge
\frac{1}{d!}\sum_{j=\lceil K/3\rceil}^{\lfloor 2K/3 \rfloor} j^{d-1}(2K-2j) \\
\ge &\frac{1}{d!}\frac{K}{3} \left(\frac{K}{3}\right)^{d-1}\frac{2K}{3}&\ge (6d)^{-(d+1)} (\hat{y}_0 - \hat{x}_0)^{d+1}
\end{eqnarray*}
proving \eqref{eq:hat_M_estimate1} with $c_1 = (6d)^{-(d+1)}$.

Let $\phi$ be the shift-invariant nearest neighbour interaction corresponding to $M_0$. We will show that there exists a suitable constant $c_2 >0$ (depending on $M_0$ and $d$) so that
$|M_0(x,y)| \le c_2 N^d$. By expressing $M_0$ in terms of its interaction we see that
$$|M_0(x,y)| \leq \sum_{C \cap D_N \ne \emptyset} |\phi(y|_C) - \phi(x|_C)|\leq \sum_{n \in D_N} \sum_{C \cap \{n\} \ne \emptyset} |\phi(x|_C) - \phi(y|_C)|,$$
where $C$ ranges over all the cliques (edges and vertices) in $\mathbb{Z}^d$. It follows that
$|M_0(x,y)| \le c_2'|D_N|$
where
$$c_2' := (4d+2)\sup\left\{|\phi(x|_C)|~:~ x \in X_r\text{ and }C \text{ is a clique in }\Z^d\right\}.$$
Since $|D_N| \le (2N+1)^d$, it follows that
$|M_0(x,y)| \le c_2N^d$ with $c_2:=3^dc_2'$.

Putting everything together, we conclude that
$$M(x,y) \ge \alpha\hat{M}(x,y) -|M_0(x,y)| \ge c_1\alpha (\hat{y}_0-\hat{x}_0)^{d+1}- c_2\cdot N^{d}.$$

\end{proof}
Under the same hypothesis except with $\alpha<0$, we get,
$$M(x,y)\leq c_1\alpha \cdot (\hat{y}_0-\hat{x}_0)^{d+1}+ c_2 N^d$$ for the same constants $c_1, c_2>0$.

\begin{lemma}\label{lem:trivialsupport}
Let $\mu$ be a shift-invariant measure on $X_r$ and $n \in \mathbb{Z}^d$ such that $\|n\|_1=1$.
If
$$\left| \int \DHt(x,n)d\mu(x) \right| =1,$$
then $\mu$ is frozen.
\end{lemma}

\begin{proof}
If $\left| \int \DHt(x,n)d\mu(x) \right| =1$ then either $\mu(\{ x \in X_r~:~ x_0 - x_{n} = 1 \mod r\})=1$ or $\mu(\{ x \in X_r~:~ x_0 - x_{n} = -1 \mod r\})=1$.
In the first case it follows that $\mu$-almost surely $x_{m-n}=x_m+1 \mod r$ and $x_{m+n}=x_m-1\mod r$ for all $m \in \mathbb{Z}^d$, so $\mu$-almost surely $x$ is frozen.
The second case is similar.
\end{proof}

In the course of our proof, it will be convenient to restrict to ergodic shift-invariant Markov random fields.
The following claim justifies this:
\begin{thm} \label{lem:MRF_ergodic_decomp}
All shift-invariant Markov random fields $\mu$ with specification $\Theta$ are in the closure of the convex hull of the
ergodic shift-invariant Markov random fields with specification $\Theta$.
\end{thm}
\begin{proof}
See Theorem $14.14$ in \cite{Georgii}.
\end{proof}

We now proceed to complete the proof of Proposition \ref{prop:non_gibbs_frozen}.
\begin{proof}
Since a convex combination of frozen measures is frozen, by Theorem \ref{lem:MRF_ergodic_decomp} it suffices to prove that any ergodic Markov random field $\mu$ adapted to $X_r$ with its Radon-Nikodym cocycle equal to $e^M$ on its support where $M=M_0 + \alpha \hat{M}$ (as in Corollary \ref{cor:X_r_cocycle_decomp}) and $\alpha \ne 0$ is frozen.

Choose any $\mu$ satisfying the above assumptions, assuming without loss of generality that $\alpha >0$. Let
\begin{equation}\label{eq:v_j_def}
v_j:= \int \DHt(x,e_j) d\mu(x) \mbox{ for } j=1,\ldots,d.
\end{equation}
If $|v_j| = 1$ for some $1 \le j \le d$, it follows from Lemma \ref{lem:trivialsupport} that $\mu$ is frozen. We can thus assume that $|v_j| <1$ for all $1 \le j \le d$. Choose $\epsilon >0$ satisfying $\epsilon < \frac{1}{4}\min\{ 1- |v_j|~:~ 1 \le j \le d\}$.

For $k \in \N$, let 

$$ A_k = \left\{ x \in X_r ~:~ \sup_{ \|w\|_1=k}\frac{1}{k}\left| \DHt(x,w) - \langle w,v\rangle\right| < \epsilon \right\}.$$
By Lemma \ref{lem:height_slope_limit} for sufficiently large $k$, $\mu(A_k)> 1- \epsilon$.
Consider $x \in A_k\cap supp(\mu)$ such that
$\mu(A_k \mid [x]_{\partial D_{k-1}}) > 1-2\epsilon$.
Then for all $y \in A_k$ satisfying $y|_{D_{k-1}^c}=x|_{D_{k-1}^c}$ and $n \in \partial D_{k-1}$
\begin{equation}\label{eq:x_low}
-\DHt(y,n) = \hat{y}_0 - \hat{y}_n \le -\langle n,v\rangle + \epsilon k < (1-\epsilon)k.
\end{equation}
Choose $\hat{z}$ which is maximal in $Ht_{\hat{x},D_{k-1}}$ and let $z = \phi_r(\hat{z})$.
It follows from Lemma \ref{lem:heightmaximum} that for some $n \in \partial D_{k-1}$,
\begin{equation}\label{eq:y_high}
\hat{z}_0 - \hat{x}_n = \|n\|_1=k.
\end{equation}

Since $x|_{D_{k-1}^c} = y|_{D_{k-1}^c} = z|_{D_{k-1}^c}$ we can assume by Lemma \ref{lem:homoclinic_heights} that
\begin{equation}\label{eq:x_y_z_same_n}
\hat{x}|_{D_{k-1}^c}=\hat{y}|_{D_{k-1}^c}=\hat{z}|_{D_{k-1}^c}.
\end{equation}
\eqref{eq:x_low} together with \eqref{eq:y_high} and \eqref{eq:x_y_z_same_n} imply that $\hat{z}_0 -\hat{y}_0 > k \epsilon$ for all $ y \in A_k$ satisfying $y|_{D_{k-1}^c}= z|_{D_{k-1}^c}$. Thus, by Lemma \ref{lem:bigmeasure}
$$M(y,z) > c_1\alpha(k\cdot \epsilon)^{d+1}-c_2k^d > c_3 k^{d+1},$$
the last inequality holding for some $c_3>0$, when $k$ is sufficiently large.

It follows that
$$ \mu([z]_{D_{k-1}} \mid [x]_{\partial D_{k-1}}) \ge \mu([y]_{D_{k-1}} \mid [x]_{\partial D_{k-1}})e^{c_3 k^{d+1}}.$$
We can write $A_k \cap [x]_{\partial D_{k-1}} = \bigcup_y ([y]_{D_{k-1}} \cap [x]_{\partial D_{k-1}})$, where the union is over all $y \in\B_{D_{k-1}}(X_r)$ such that $[y]_{D_{k-1}}\cap A_k \cap [x]_{\partial D_{k-1}}\neq \emptyset$.
There are at most $|\mathcal{B}_{D_{k-1} \cup \partial D_{k-1}}(X_r)| = e^{O(k^d)}$ terms in the union above so
$$ \mu(A_k \mid [x]_{\partial D_{k-1}}) \le e^{O(k^d)}e^{-c_3 k^{d+1}} \mu([z]_{D_{k-1}} \mid [x]_{\partial D_{k-1}}) .$$
It follows that $\mu(A_k \mid [x]_{\partial D_{k-1}}) \to 0$ as $k \to \infty$. For $k$ sufficiently large this would contradict our choice of $x$, for which $\mu(A_k \mid [x]_{\partial D_{k-1}}) > 1-2\epsilon$.
\end{proof}

\section{Non-frozen adapted shift-invariant Markov random fields on $X_r$ are fully-supported}\label{non_frozen_adapted_MRF_X_r_fully_supported}
We have concluded that any shift-invariant Markov random field which is adapted with respect to $X_r$ is a Gibbs measure for some shift-invariant nearest neighbor interaction. Our next goal is to show that any such measure must be fully-supported on $X_r$.

\begin{prop}\label{prop:MRF_X_r_full_or_frozen}
Let $r \ne 1,2,4$ be a given positive integer, and
let $\mu$ be a shift-invariant Markov random field adapted with respect to $X_r$. Then either $supp(\mu)=X_r$ or $\mu$ is frozen.
\end{prop}

Roughly speaking we shall show that for non-frozen shift-invariant Markov random fields the height function corresponding to a typical point is ``not very steep''. Given a height function that is ``not very steep'', there is enough flexibility to ``deform'' the height function while keeping the values fixed outside some finite set. 
For an adapted Markov random field, the ``deformed height function'' corresponds to a point in the support as well, which will be the key to proving the required result. Somewhat related methods can be found in Section 4.3 of \cite{schmidt_cohomology_SFT_1995}.

We first introduce some more notation. For $x \in X_r$ and a finite set $F \subset \Z^d$ denote:
\begin{equation}
Range_F(x) := \max_{n \in F} \DHt(x,n) - \min_{n \in F} \DHt(x,n).
\end{equation}


Given $A \subset \ZD$, $ \hat{x} \in Ht(A)$ and a finite set $F \subset A \subset \Z^d$, we define:
\begin{equation}
Range_F(\hat{x}) := \max_{n \in F}\hat{x}_n- \min_{n \in F}\hat{x}_n.
\end{equation}
It follows that if $\hat{x} \in Ht$ and $x \in X_r$ are such that $x=\phi_r(\hat{x})$ then for all finite $F\subset \ZD$, $Range_F(x) = Range_F(\hat{x})$.


\begin{lemma}\textbf{(``Extremal values of height obtained on the boundary'')}\label{lem:steeptoflat}
Let $F \subset \mathbb{Z}^d$ be a finite set and $\hat{x} \in Ht$
such that $Range_{\partial F}(\hat x) > 2$. Then there exists $\hat{y} \in Ht$ such that $\hat{y}_n =\hat{x}_n $ for all $n \in F^c$ and
\begin{equation*}
Range_{ F}(\hat{y}) = Range_{\partial F}(\hat{y})-2= Range_{\partial F}(\hat{x})-2.
\end{equation*}
\end{lemma}
\begin{proof}
Denote $$T:=\max_{ n \in \partial F}\hat{x}_n \mbox{ and } B:=\min_{n \in \partial F}\hat{x}_n .$$

Let
$$ \kappa = \kappa(\hat x,F) := \sum_{n \in F} \max(\hat x_n - T+1, B-\hat x_n+1,0).$$

The number $\kappa$ is the absolute value for the deviations of $\hat{x}|_F$ from the (open) interval $(B,T)$.
We prove the claim by induction on $\kappa$. If $\kappa=0$ then $y=x$ already satisfies the conclusion of this lemma because
$B+1 \le \hat{x}_n \le T-1$ for all $n \in F$ which implies that
$$Range_F(\hat x) = \max_{m \in F}\hat{x}_m - \min_{m \in F }\hat{x}_m \le (\max_{m \in \partial F}\hat{x}_m -1) - (\min_{m \in \partial F }\hat{x}_m +1)$$
$$ = Range_{\partial F}(\hat x) -2.$$

Now suppose $\kappa > 0$, and let $n \in F$ be a coordinate where $\hat{x}$ obtains an extremal value for $F \cup \partial F$. Without loss of generality suppose,
\begin{equation*}
\displaystyle{\hat{x}_n= \max_{m \in F \cup \partial F}\hat{x}_m}.
\end{equation*}
Since all neighbors of $n$ are in $F \cup \partial F$, it follows that $\hat{x}_m=\hat{x}_n -1$ for all $m$ adjacent to $n$.
Therefore we have $\hat{y} \in Ht$ given by
\begin{eqnarray*}
\hat{y}_m:=\begin{cases}\hat{x}_m-2&\text { for } m= n\\
\hat{x}_m& \text{ otherwise}.\end{cases}
\end{eqnarray*}
Since $Range_{\partial F}({\hat x}) >2$,
it follows that $\hat{y}_n$ is neither a minimum nor a maximum for $\hat{y}$ in $F \cup \partial F$. Thus $\kappa(\hat y,F) < \kappa(\hat x,F)$ and so we can apply the induction hypothesis on $\hat y$ and conclude the proof.
\end{proof}

\begin{lemma}\textbf{(``Flat extension of an admissible pattern'')}\label{lem:flat_outside}
Let $\hat{x} \in Ht$ and $N \in \N $. Then there exists $\hat{y} \in Ht$ such that $\hat{y}_n=\hat{x}_n$ for $n \in D_{N+1}$ and
\begin{equation*}
Range_{\partial D_{N+k}}(\hat{y})=Range_{\partial D_N}(\hat{x})-2k,
\end{equation*}
for all $1\leq k \leq \frac{Range_{\partial D_N}(\hat x)}{2}$.
\end{lemma}

\begin{proof}

We will prove the following statement by induction on $M \in \mathbb{N}$: For all $N \in \mathbb{N}$ and height functions $\hat{x} \in Ht(D_{N+1+M})$ with
$Range_{\partial D_N} (\hat{x}) = 2M$ there exists a height function $\hat{y} \in Ht(D_{N+1+M})$ such that $\hat{y}_n=\hat{x}_n$ for all $n \in D_{N+1}$ and $1\le k \le M$, $Range_{\partial D_{N+k}}(\hat y) = 2M -2k$. Observe that the height function $\hat{y}$ satisfies in particular $Range_{\partial D_{N+M}}(\hat y) = 0$. Thus, the outermost boundary of $\hat{y}$ is flat and it can be extended to a height function on $\mathbb{Z}^d$, so the lemma will follow immediately once we prove the statement above for all $M \in \mathbb{N}$.

For the base case of the induction, there is nothing to prove.

Assume the result for some $M \in \N$. Let $\hat{x} \in Ht$ be a height function such that $Range_{\partial D_N} (\hat{x}) = 2(M+1)$.
Denote $\tilde{N}:=N+1+(M+1)=N+M+2$. Let $n \in D_{\tilde{N}}\setminus D_{{N+1}}$ be a site where $\hat{x}$ obtains an extremal value for $D_{\tilde{N}}\setminus D_{N}$. If there is no such site then
\begin{eqnarray*}
Range_{\partial D_{N+1}} (\hat{x})&=& \max_{m \in \partial D_{N+1}}\hat{x}_m-\min_{m \in \partial D_{N+1}}\hat{x}_m\\&=& (\max_{m \in \partial D_{N}}\hat{x}_m-1)-(\min_{m \in \partial D_{N}}\hat{x}_m+1)\\&=&2M
\end{eqnarray*} proving the induction step for that case.
Without loss of generality we assume that it is a maximum, that is,
\begin{equation*}
\hat{x}_n= \max_{m\in D_{\tilde{N}}\setminus D_{N}} \hat{x}_m.
\end{equation*}

Then the function $\hat{\tilde{y}}$ given by
\begin{equation*}
\hat{\tilde{y}}_m= \begin{cases}\hat{x}_n-2 &\text{ if } m=n\\ \hat{x}_m&\text{ otherwise }\end{cases}
\end{equation*} is a valid height function on $D_{\tilde{N}}$. Hence we have lowered the height function at the site $n$ of $\hat{x}$. Repeating the steps for sites with extremal height (formally, this is another internal induction, see for example the proof of Lemma \ref{lem:steeptoflat}), a height function $\hat{z}$ can be obtained on $D_{\tilde{N}}$ such that $\hat{z}= \hat{x}$ on $D_{N+1}$ and
\begin{equation*}
Range_{\partial D_{{N+1}}} (\hat{z})=2M.
\end{equation*}
Thus we can apply the induction hypothesis to $\hat{z}$, substituting $N+1$ for $N$ to obtain a height function $\hat{y}$ on $D_{\tilde{N}}$ such that $\hat{y}= \hat {x}$ on $D_{N+1}$ and
\begin{equation*}
Range_{\partial D_{N+k}}(\hat{y})=Range_{\partial D_{N}}(\hat{x})-2k
\end{equation*}
for $1\leq k \leq \frac{Range_{\partial D_{N}}(\hat{x})}{2}.$

This completes the proof of the statement.
\end{proof}
\begin{lemma}\label{lem:notsteeptoanything}\textbf{(``Patching an arbitrary finite configuration inside a non-steep point'')}
Let $r\ne1,2,4$ be a positive integer and $N, k \in \N$. Choose $y \in X_r$ which satisfies $Range_{\partial D_{2N+2r+k+1}}(y)\leq 2k$ and some $x \in X_r$. Then:
\begin{enumerate}
\item{If either $r$ is odd or $x_n -y_n$ is even for all $n \in \mathbb{Z}^d$, then there exists $z \in X_r$ such that
\begin{equation*}
z_n=\begin{cases} x_n &\text{ if } n \in D_N\\y_n &\text { if }n \in D_{2N+2r+k+1}^c\end{cases}
\end{equation*} }
\item{If $r$ is even and $x_n -y_n$ is odd for all $n \in \mathbb{Z}^d$
, then there exists $z \in X_r$ such that

\begin{equation*}
z_n=\begin{cases} x_{n+e_1} &\text{ if } n\in D_N\\y_n &\text { if }n \in D_{2N+2r+k+1}^c\end{cases}
\end{equation*}
}

\end{enumerate}
\end{lemma}
The idea of this proof lies in the use of Lemmata \ref{lem:steeptoflat} and \ref{lem:flat_outside}. Given any configuration on $D_N$ we can extend it to a configuration on $D_{2N}$ with flat boundary which can be then extended to $y$ a little further away provided the range of $y$ is not too large.
There is a slight technical issue we deal with in the proof below in case $r$ is even. This is essentially due to the fact that for $r$ even we have
$$x_n -x_m \equiv \| n- m \|_1 \mod 2 \mbox{ for all } x \in X_r \mbox{ and } n,m \in \ZD.$$
\begin{proof}
By applying Lemma \ref{lem:steeptoflat} $k-1$ times we conclude that there exists $ y'\in X_r$ such that $y'=y$ on $D_{2N+2r+k+1}^c$ and
\begin{equation*}
Range_{\partial D_{2N+2r+2}}(y')=2.
\end{equation*}

For all $n,m \in \partial D_{2N+2r+2}$, $||n -m||_1$ is even, therefore $\hat y'_{n}$ and $\hat y'_{m}$ have the same parity. Thus, there exist two values $c,d\in \Z_r$ which $y'$ takes on $\partial D_{2N+2r+2}$. Consider $a\in \Z_r$ which is adjacent (for the Cayley graph of $\Z_r$) to both $c$ and $d$. Then the configuration $y^{(1)}$ given by

\begin{equation*}
y^{(1)}_n:=
\begin{cases}
y'_n& \text{ if }\|n\|_1 \geq 2N+2r+3\\
a&\text{ if } \|n\|_1\leq 2N+2r+2 \text{ is even}\\
c&\text{ if } \| n\|_1 \leq 2N+2r+2 \text{ is odd}
\end{cases}
\end{equation*}
is an element of $X_r$. By Lemma \ref{lem:flat_outside} choose $x^{(1)}\in X_r$ such that $x^{(1)}= x$ on $D_N$ and
\begin{equation*}
Range_{\partial D_{2N-1}}(x^{(1)})=0.
\end{equation*}

Equivalently, there exists $b \in \mathbb{Z}_r$ so that $x^{(1)}_n = b$ for all $n \in \partial D_{2N-1}$.

If we are in case $(1)$, either $r$ is even and $b \equiv a \mod 2$ or $r$ is odd. Either way, there is some integer $k \in [0,\ldots r-1]$ such that $a +k \equiv b- k \mod r$. Thus, we can find $y^{(2)},x^{(2)} \in X_r$, so that $y^{(2)}$ agrees with $y^{(1)}$ in $D_{2N+2r}^c$, $x^{(2)}$ agrees with $x^{(1)}$ in $D_{2N}$, and so that both $x^{(2)}$ and $y^{(2)}$ have a common constant value $a+ k = b + (r-k) \mod r$ on $\partial D_{2N+r-k-1}$. Thus, we get the required $z \in X_r$ by setting

\begin{equation*}
{z}_n :=\begin{cases}{x}^{(2)}_n \text { for }n\in D_{2N+r-k-1}\\{y}^{(2)}_n \text{ for }n\in D_{2N+r-k-1}^c\end{cases}
\end{equation*}

To prove case $(2)$ we follow the same procedure, substituting $x$ by $\sigma_{e_1} (x)$.\end{proof}

We can now conclude the proof of Proposition \ref{prop:MRF_X_r_full_or_frozen}:

\begin{proof}
Let $\mu$ be a shift-invariant Markov random field and $v_1,\ldots,v_d$ be given by \eqref{eq:v_j_def}.

Assume that $supp(\mu)$ is not frozen. Then by Lemma \ref{lem:trivialsupport}, $|v_j| <1$ for all $1 \le j \le d$. Again, choose $\epsilon >0$ satisfying $\epsilon < \frac{1}{4}\min\{ 1- |v_j|~:~ 1 \le j \le d\}$.

We need to show that for all $N \in \N$ and configurations $c \in \mathcal{B}_{D_N}(X_r)$, $\mu([c]_{D_N})>0$.

From Lemma \ref{lem:height_slope_limit} it follows that for sufficiently large $k$,
$$\mu( \{ y \in X_r ~:~ Range_{\partial D_k}(y) \le 2(1-\epsilon)k\})>1-\epsilon.$$ Now choose $k > (2N+2r+1)$ large enough so that there exists $y \in supp(\mu)$ with
$$Range_{\partial D_k}(y) \le 2(1-\epsilon)k\le 2(k-(2N+2r+1)).$$
By Lemma \ref{lem:notsteeptoanything}, it follows that there exists $z \in X_r$ with $z_n = y_n $ for $n \in \Z^d \setminus D_k$ and $z_n = c_n$ for $n \in D_N$. Since $\mu$ is an adapted Markov random field it follows that $z\in supp(\mu)$, in particular $\mu([c]_{D_N})>0$.

\end{proof}

\section{Fully supported shift-invariant Gibbs measures on $X_r$ }\label{full support 3cb exists}
Next we demonstrate the existence of a fully-supported shift-invariant Gibbs measure for shift-invariant nearest neighbor interactions on $X_r$. We will obtain such measures by showing that equilibrium measures for certain interactions are non-frozen and thus are fully supported by Proposition \ref{prop:MRF_X_r_full_or_frozen}. To state and prove this result, we need to introduce (measure-theoretic) pressure and equilibrium measures and apply a theorem of Lanford and Ruelle relating equilibrium measures and Gibbs measures. Our presentation is far from comprehensive, and is aimed to bring only definitions necessary for our current results. We refer readers seeking background on pressure and equilibrium measures to the many existing textbooks on the subject, for instance \cite{Rue,walters-book}.

Let $\mu$ be a shift-invariant probability measure on a shift of finite type $X$.
The \emph{measure theoretic entropy} can be defined by
\begin{equation}
h_\mu :=\lim_{N \rightarrow \infty}\frac{1}{|D_N|}H^{D_N}_{\mu},
\end{equation}

where $D_N$ was defined in \eqref{eq:D_N_def} and
\begin{equation}
H^{D_N}_{\mu}:=\sum_{a\in \B_{D_N}(X)}-\mu([a]_{D_N})\log{\mu([a]_{D_N})},
\end{equation}
with the understanding that $0\log 0=0$.


Given a continuous function $f:X \to \mathbb{R}$, the \emph{measure-theoretic pressure} of $f$ with respect to $\mu$ is given by
$$P_\mu(f) := \int f d\mu + h_\mu.$$
A shift-invariant probability measure $\mu$ is an \emph{equilibrium state} for $f$ if the maximum of $\nu \mapsto P_\nu(f)$ over all shift-invariant probability measures is attained at $\mu$. The existence of an equilibrium state for any continuous $f$ follows from upper-semi-continuity of the function $\nu \mapsto P_\nu(f)$ with respect to the weak-$*$ topology.

Let $\phi$ be a nearest neighbor interaction on X. As in \cite{Rue} define a function $f_\phi:X\longrightarrow \R$ by
\begin{equation}\label{eq:f_phi_def}
f_\phi(x) := \sum_{A\text{ finite }~:~0\in A \subset \Z^d}\frac{1}{|A|}\phi(x|_{A}).
\end{equation}

Note that there are only a finite number of non-zero summands in \eqref{eq:f_phi_def} above, because  $\phi$ is a nearest neighbor interaction

The following is a restricted case of a classical theorem by Lanford and Ruelle:
\begin{thm*}\textbf{(Lanford-Ruelle Theorem \cite{lanfruell, Rue})} Let $X$ be a $\Z^d$-shift of finite type and $\phi$ a shift-invariant nearest neighbor interaction. Then any equilibrium state $\mu$ for $f_\phi$ is a Gibbs state for the given interaction $\phi$. \label{equiGibbs}
\end{thm*}The \emph{topological entropy} of a $\Z^d$-subshift $X$ is given by
$$h(X) := \lim_{k \to\infty}\frac{1}{|D_k|}\log|\mathcal{B}_{D_k}(X)|.$$

We recall the well known \emph{variational principle} for topological entropy of $\Z^d$-actions, which (in particular) asserts that $h(X)=\sup_{\nu}h_\nu$ whenever $X$ is a $\Z^d$-shift space and the supremum is over all probability measures on $X$.

\begin{lemma}\label{lem:forzen_zero_entropy}
Let $\mu$ be a shift-invariant, frozen Markov random field on $\mathcal{A}^{\ZD}$, then $h_\mu =0$.
\end{lemma}

\begin{proof}
Consider $X_\mu := supp(\mu)$. This is a shift-invariant topological Markov field, consisting of frozen points.
Thus for all finite $F \subset \ZD$, $|\mathcal{B}_F(X_\mu)| \le |\mathcal{B}_{\partial F} (X_\mu)|$ . In particular,
$$ \log |\mathcal{B}_{D_k}(X_\mu)| \le \log |\mathcal{B}_{\partial D_k}(X_\mu)| \le C k^{d-1}.$$
It follows that $h(X_\mu)= 0$, so by the variational principle $h_\mu =0$.
\end{proof}

\begin{lemma}\label{lem:equilibrium_non_frozen}
Let $M$ be a Gibbs cocycle on $X_r$ with a shift-invariant nearest neighbor interaction. Then there exists a shift-invariant nearest neighbor interaction $\phi$ such that $M=M_\phi$ and any equilibrium measure for $f_\phi$ is non-frozen.
\end{lemma}
\begin{proof}

Let $(x^{(i)},y^{(i)}) \in \Delta_{X_r}$ be as in the proof of Proposition \ref{prop:X_r_markov_cocycles}. If $M \in \mathcal{G}_{X_r}$ then there exists a shift-invariant nearest neighbor interaction $\phi$ so that
\begin{eqnarray*}
M(x^{(i)},y^{(i)})&=&\phi([i+2]_0)-\phi([i]_0) \\
&+& \sum_{j=1}^{d} \left(\phi([i+2,i+1]_j)-\phi([i+1,i]_j) + \phi([i+1,i+2]_j) - \phi([i,i+1]_j)\right).\end{eqnarray*}

Consider the nearest neighbour interaction $\tilde \phi$ given by
\begin{equation*}
\tilde \phi \left( [i+1,i]_j\right) := \tilde \phi \left( [i,i+1]_j\right) :=
\frac{1}{2d} \left[
\sum_{k=1}^d \left(\phi([i+1,i]_k)+\phi([i,i+1]_k) \right) + \phi([i]_0)+ \phi([i+1]_0)
\right]
\end{equation*}
for all $i \in \Z_r$ and $j \in \{1,\ldots,d\}$, and $\tilde \phi([i]_0)=0$ for all $i \in \Z_r$.

It follows 
that $M(x^{(i)},y^{(i)}) = M_{\tilde \phi}(x^{(i)},y^{(i)})$ for all $i \in \Z_r$ and so $M = M_{\tilde \phi}$.

Thus we can assume without loss of generality that $\phi = \tilde \phi$ satisfies
$$\phi([i,i+1]_j)=\phi([i+1,i]_j)=a_i \mbox{ for all } i \in \Z_r \mbox{ and }j \in\{1, 2, \ldots, d\}$$
and $\phi([i]_0)=0$ for all $i \in \Z_r$.

By \eqref{eq:f_phi_def}:
$$ \int f_\phi(x) d\mu(x) = \int \frac{1}{2} \sum_{j=1}^d\left( \phi([x_{-e_j},x_{0}]_j) + \phi([x_0,x_{e_j}]_j)\right) \mu(x)$$

Thus:
\begin{eqnarray*}
\int f_\phi(x) d\mu(x) &=&\sum_{j=1}^d\sum_{i=0}^{r-1} \phi([i,i+1]_j)\mu([i,i+1]_j)+\phi([i+1,i]_j)\mu([i+1,i]_j)\\
&=& \sum_{j=1}^{d}\sum_{i=0}^{r-1} a_i(\mu([i,i+1]_j)+\mu([i+1,i]_j)).
\end{eqnarray*}


Let $a = \max_{0 \le i \le r-1} a_i$ attained by $a_{i_0}$. It follows that for any shift-invariant probability measure
$$\int f_\phi(x) d\mu(x) \le d\cdot a$$
with equality holding iff $\mu ([i,i+1]_j)=\mu ([i+1,i]_j)=0$ for all $a_i <a$ and $j=1,\ldots,d$.

For a frozen measure $\mu $ it follows that for some $j \in \{1, 2, \ldots, d\}$,
$\mu([i, i+1]_j)>0$ for all $i \in \{0, 1,\ldots, r-1\}$ or $\mu([i+1, i]_j)>0$ for all $i \in \{0, 1,\ldots, r-1\}$. Thus if $a_i <a$ for \emph{some } $0 \le i \le r-1$, it follows that for any frozen measure $\mu$,
\begin{equation}\label{eq:phi_int_not_max}
\int f_\phi(x) d\mu(x) < \sup_{\nu} \int f_\phi(x) d\nu(x).
\end{equation}
where the supremum is attained by the measure supported on the orbit of the periodic point $x \in X_r$ given by
$$x_n := \left\{
\begin{array}{ll} i_0 & \|n\|_1 \mbox{ odd}\\
i_0+1 & \|n\|_1 \mbox{ even }
\end{array}\right.
$$

By Lemma \ref{lem:forzen_zero_entropy}, if $\mu$ is frozen then $h_\mu =0$.

Thus in this case by  \eqref{eq:phi_int_not_max} for any frozen probability measure $\mu$
$$P_\mu(f_\phi) = \int f_\phi(x) d\mu(x) < \sup_{\nu} \int f_\phi(x) d\nu(x) \le \sup_{\nu} P_{\nu}(f_\phi)$$
and in particular any frozen measure $\mu$ can not be an equilibrium measure for $f_\phi$.

The remaining case is when $a_i=a$ for all $i$, in which case $f_\phi(x)=d\cdot a$ is constant. Thus, by the variational principle $\sup_{\nu} P_{\nu}(f_\phi)= d\cdot a + \sup_{\nu} h_\nu = d\cdot a+ h(X_r)$. Since $h(X_r) >0$, it follows that the strict inequality $P_{\mu}(f_\phi) < \sup_{\nu}P_{\nu}(f_\phi)$ holds also in this case for any frozen measure $\mu$. Thus we have the result that for a given Gibbs cocycle with a shift-invariant nearest neighbor interaction there exists an interaction for that cocycle such that the corresponding equilibrium state is not frozen.
\end{proof}

\begin{corollary}
For all shift-invariant Gibbs cocycles $M$ on $X_r$ there exists a shift-invariant nearest neighbor interaction $\phi$ on $X_r$ with $M = M_\phi$ and a corresponding shift-invariant Gibbs state $\nu$ with $supp(\nu)=X_r$.
\end{corollary}
\begin{proof}
By Lemma \ref{lem:equilibrium_non_frozen}, there exists a shift-invariant nearest neighbor interaction $\phi$ on $X_r$ with $M = M_\phi$ and an equilibrium measure $\mu$ for $f_\phi$ which is non-frozen. By the Lanford-Ruelle Theorem such $\mu$ is a Gibbs state for $\phi$ and by Proposition \ref{prop:MRF_X_r_full_or_frozen} it is fully supported.
\end{proof}

\section{``Strongly'' non-Gibbsian shift-invariant Markov random fields}
\label{section: nonGibb}

In this section we describe some  shift-invariant Markov random fields whose specification is not given by any shift-invariant finite range interaction. Our example also proves that generally the specification of a shift-invariant Markov random field cannot be ``given by a finite number of parameters''. The construction is somewhat similar to the checkerboard island system as introduced in \cite{quas2003entropy}.

Let $\A$ be the alphabet consisting of the $18$ `tiles' illustrated in Figure \ref{figure: alphabet of shift space}: A blank tile, a ``seed'' tile (marked with a ``$0$''), $8$ ``interior arrow tiles''( $4$ of them have arrows in the coordinate directions and the rest are ``corner tiles'') and finally 8 additional ``border arrow tiles''( marked by an extra symbol ``B").
\begin{figure}[H]
\centering
\includegraphics[angle=0,
width=.5\textwidth]{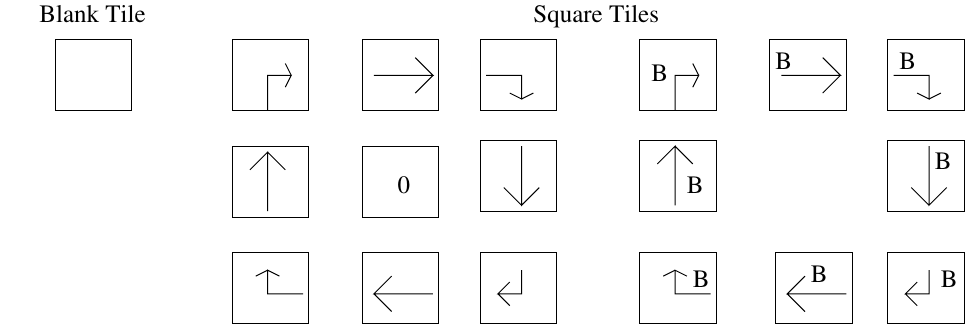}
\caption{The Alphabet $\A$}\label{figure: alphabet of shift space}
\end{figure}

All tiles other than the blank tile will be called \emph{square tiles} and the tiles with arrows will be called \emph{arrow tiles}. The arrow tiles with `B' will be called \emph{border tiles} and those without `B' will be called \emph{interior tiles}. Configurations of an $(2n+1) \times (2n+1)$ square shape
whose inner boundary consists of border tiles as illustrated by the example in Figure \ref{figure:square} will be called an \emph{$n$-square-island}. A square-island refers to an $n$-square-island for some $n$. The border tiles form the four \emph{sides} of the square-island which surround the square formed by the interior tiles.
\begin{figure}[H]
\centering
\includegraphics[angle=0,
width=.2\textwidth]{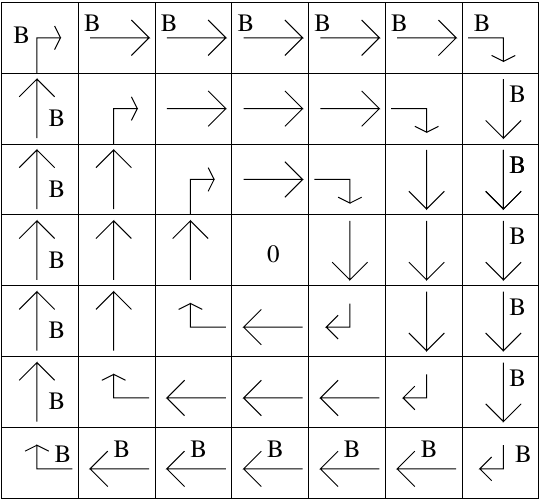}
\caption{A $3$-square-island}\label{figure:square}
\end{figure}
\noindent Informally, the idea is to have the square tiles form square-islands with the seed tile in the center and border tiles on their boundary ``floating in the sea of the blank tiles''. A `generic' configuration can be seen in Figure \ref{figure:generic configuration}.
\begin{figure}[H]
\centering
\includegraphics[angle=0,
width=.4\textwidth]{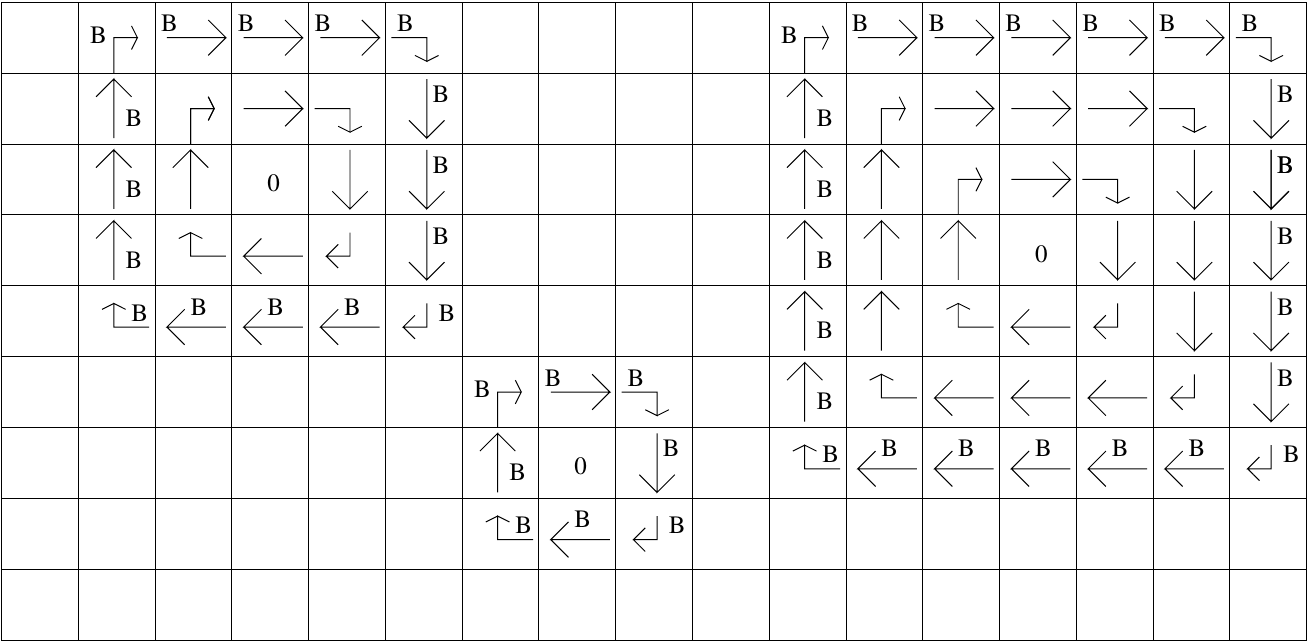}
\caption{A `Generic' Configuration}\label{figure:generic configuration}
\end{figure}

Let $X$ be the nearest neighbor shift of finite type on the alphabet $\A$ with constraints:
\begin{enumerate}
\item Any ``arrow head'' must meet an ``arrow tail'' of matching  type (interior or border) and vice versa.\label{item: arrows match}
\item Adjacent arrow tiles should not point in opposite directions.\label{item: adjacent arrows don't contradict}
\item Two corner direction tiles cannot be adjacent to one another.\label{item: corner tiles cannot sit next to each other}
\item The seed tile is only allowed to sit adjacent to straight arrow
tiles. \label{item: seed tiles sit beside arrows}
\item An interior tile is always surrounded by other square tiles while a border tile
has an interior tile on its right and the blank tile on its left (left and right here are taken from the point of view of the arrow).
\label{item: two adjacent squares not possible}
\end{enumerate}

Notice that the arrow tiles can turn only in the clock-wise direction. By Constraint (\ref{item: arrows match}), every arrow must either be a part of a bi-infinite path or a closed path. Any such path must either trace a straight line (vertical or horizontal) or an ``L shape'' or a ``U shape'', or a closed path which traces a rectangle.
By Constraint (\ref{item: two adjacent squares not possible}) we find that tiles forming a rectangular path must have square tiles to their right (in the interior of the rectangle). These are confined to the interior of the rectangle and thus must themselves trace a smaller rectangle (smaller in terms of the area confined by the rectangle).
Note that in constraint (\ref{item: two adjacent squares not possible}) we mean in particular that
corner direction border tiles  have blank tiles on the two sites on their left (where left is taken from the point of view of both the initial direction and the final direction of the arrow).
 By Constraints (\ref{item: corner tiles cannot sit next to each other}) and (\ref{item: adjacent arrows don't contradict}) we can conclude by induction on the length that closed paths must actually trace squares with a seed in the center. It follows that the finite connected components of the square tiles are square-islands. The length of these square-islands is $2N+1$ for some $N\in \N$ because of Constraints (\ref{item: corner tiles cannot sit next to each other}) and (\ref{item: seed tiles sit beside arrows}). Two such square-islands cannot be adjacent because of Constraint (\ref{item: two adjacent squares not possible}).

We can also exclude the possibility of a ``U shaped'' path using Constraints (\ref{item: corner tiles cannot sit next to each other}) and (\ref{item: adjacent arrows don't contradict}), again by induction on the length of the ``base of the U''.

For $r\in \N$ denote by
$B_r\subset \Z^2$, the $l^{\infty}$-ball of radius $r$ in $\Z^2$, that is,
$$B_r:=\{(i,j)\in \Z^2~:~ |i|, |j| \leq r\}.$$

Because a $\partial B_r$-configuration of blank tiles can be filled either with a square-island or with blank tiles, it follows that $X$ has positive entropy. 

\begin{prop}\label{prop:island_mme}
Let $\mu$ be any measure of maximal entropy for $X$. Then $\mu$ is fully-supported.
\end{prop}
\begin{proof}
We will begin by proving that $\mu$-almost surely any square tile is part of a square-island.
Since a seed must be surrounded by arrow tiles, it suffices to prove that any arrow tile is part of a square-island. By the discussion above, any arrow tile is part of a path which is bi-infinite or traces a square. An infinite path is either ``L shaped'' or a straight vertical or horizontal line. It is easily verified that the appearance of an ``L shape'' in the origin is a transient event with respect to the horizontal or vertical shifts, so by Poincar\'{e} recurrence the probability of having a ``L shaped path'' is zero. An infinite horizontal line forces a half space of horizontal straight line. This forces either a transient event or periodicity. Because $X$ has positive entropy, for a measure of maximal entropy for $X$, the measure of periodic points is 0 as well. Thus $\mu$-almost surely, any arrow tile is part of a
path which traces a square. By Constraint (\ref{item: two adjacent squares not possible}) either the square tile is contained in a square-island or there is an infinite sequence of nested square paths. The latter is again a transient event.

Suppose $x\in X$  does not have any infinite connected component composed of square tiles. We will now show for all $r\in \N$ that there exists a finite set $A_r$ such that $B_r\subset A_r \subset \Z^2$ and $x_i$ is the blank tile for all $i \in \partial A_r$.
Let $Sq_1\in \A^{C_1}, Sq_2\in \A^{C_2}, \ldots, Sq_k\in \A^{C_k} $ be an enumeration of the square-islands in $x$ such that $C_{i}\cap B_{r+1}\neq \emptyset$ . Let
\begin{equation*}
A_{r}:= \bigcup_{i=1}^k C_i \cup B_{r}.
\end{equation*}
Since every square-island is surrounded by the blank tile, $A_r$ has the required properties.

Consider some $y \in X$ and $n \in \N$. We will prove that $\mu([y]_{B_n})>0$. Any incomplete square-island in $y|_{B_r}$ can be completed (possibly in multiple ways) in $B_{4r}$. By completing these square-islands we can obtain $z\in X$ such that it satisfies
\begin{equation*}
z_i:=\begin{cases}y_i&\text{for }i \in B_r\\\text{blank tile}&\text{for } i \in B_{4r}^c. \end{cases}
\end{equation*}

Now choose any $x \in \mathit{supp}(\mu)$ which does not have any infinite connected component composed of square tiles. As previously discussed we can find $A_{4r} \subset \mathbb{Z}^2$ such that $B_{4r} \subset A_{4r}$ and $x_i$ is the blank tile for all $i \in \partial A_{4r}$.
Then $z|_{\partial A_{4r}}=x|_{\partial A_{4r}}$. By the Lanford-Ruelle Theorem $\mu$ is a Markov random field with the uniform specification adapted to $X$. Therefore
\begin{equation*}
\frac{\mu([z]_{A_{4r}\cup \partial A_{4r}})}{\mu([x]_{A_{4r}\cup\partial A_{4r}})}=1
\end{equation*}
proving $\mu([z]_{B_r})= \mu([y]_{B_r}])>0$.

\end{proof}

The following small technical lemma will be of use to us:
\begin{lemma}[Recovering a square-island]\label{lem:recover_square_island}
Let  $F_0\ \subset \Z^2$ be a finite set which is connected in the Cayley graph of $\Z^2$.
Furthermore suppose that there exist $(m_1,n_1)(m_2,n_2),(m_3,n_3) \in F_0$ such that $n_2-n_1=2i \in 2\mathbb{N}$.
Suppose $x \in X$ satisfies the following properties:
\begin{itemize}
\item[(A)] $x_v$ is a square tile for all $v \in F_0$.
\item[(B)] $x_{(m_1,n_1)}$ and $x_{(m_2,n_2)}$ are both horizontal border tiles or corner border tiles.
\item[(C)] $x_{(m_3,n_3)}$ is a vertical or corner border tile.
\end{itemize}
Then there exists $(m_4,n_4) \in \Z^2$, which depends only on $x|_{F_0}$  such that $x|_{B_i +(m_4,n_4)}$ is an  $i$-square-island.
\end{lemma}
\begin{proof} The proof is a careful but fairly straightforward inspection of the rules defining $X$, which we now sketch.
Because $F_0 \subset \Z^2$  is  connected in the Cayley graph of $\Z^2$,   it follows that
$x|_{F_0}$ is a connected pattern of square tiles thus it is a  subpattern  of a square-island $x|_S$. The fact that  $x_{(m_1,n_1)}$ and $x_{(m_2,n_2)}$
are both horizontal or corner  border tiles implies that $S$ is a finite square and also determines
the vertical position of the upper and lower sides  of $S$. Thus  $S$  must be a $(2i+1)\times(2i+1)$-square and  the vertical coordinate of its center is $n_4 = \frac{1}{2}(n_1+n_2)$.
The fact that $x_{(m_3,n_3)}$ is a vertical or corner  border tile determines the horizontal coordinate of one of the vertical sides of $S$.
This determines the horizontal coordinate $m_4$ of the center of $S$ to be either $m_3 +i$ or $m_3  -i$ according to the direction of the arrow in $x_{(m_3,n_3)}$ .
\end{proof}

We now describe a subshift of finite type $Y$ for which the space of the Markov cocycles is infinite dimensional.
$Y$ is a nearest neighbor shift of finite type with the alphabet as in Figure \ref{figure: alphabet of shift space}
but now with two types of square tiles, type 1 and 2.
The adjacency rules are as in the subshift $X,$ but also force adjacent square tiles to be of the same type, that is,
any square-island in an element of $Y$ will consist entirely of tiles of type 1 or of type 2.
We denote by $\phi: Y \to X$ the map which forgets the types of square tiles.

Define  functions $\psi_1,\psi_2:Y \to \Z_+$   by
\begin{equation}
\psi_j(y) =\begin{cases}
i &  y|_{B_i} \mbox{ is  an } i\mbox{-square-island of type } j\\
0 & \mbox{ otherwise}
\end{cases}
\mbox{ for }  j=1,2.
\end{equation}

Given  a function $p:\N \to (0,1)$ define a shift-invariant measure $\mu_p$ on $Y$ as follows:
Pick  $x\in X$ according to a fixed measure of maximal entropy $\mu$ and then choose the type of square-islands in $x$ independently,
where an $i$-square-island is of type $1$ with probability $ p(i)$.

Precisely, the measure $\mu_p$ is the unique shift-invariant Borel probability measure on $Y$ satisfying $\mu = \mu_p \circ \phi^{-1}$,
$$ \mu_p \left( \psi_1(\sigma_{v}(y))=i ~|~ \psi_1(\sigma_v(y)) +
\psi_2(\sigma_v(y))= i\right) = p(i) \text{ for } i \ge 1$$
and $(\psi_1(\sigma_v(y)))_{v \in \Z^2}$ are jointly independent
conditioned on
$(\psi_1(\sigma_v(y))+\psi_2(\sigma_v(y)))_{v \in \Z^2}$.

Let $g_p:Y \to \mathbb{R}$ be given by
\begin{equation}
g_p(y) =  \begin{cases}
\log p(\psi_1(y)) & \psi_1(y)>0\\
\log [1-p(\psi_2(y))] & \psi_2(y)>0\\
0 & \mbox{otherwise},
\end{cases}
\end{equation}
and
$M_p:\Delta_Y \to \mathbb{R}$ be given by:

\begin{equation}\label{eq:_M_p_def}
 M_\bold p(y, y')= \sum _{v \in \Z^2} \left[g_p(\sigma_v(y')) - g_p(\sigma_v(y))\right]
\end{equation}

It is clear that $M_p$ is a well-defined shift-invariant $\Delta_Y$-cocycle.


Moreover, if $y \in Y$ and $\Lambda, \Gamma \subset \Z^2$ are finite sets so that $\Lambda \cap \partial \Lambda \subset \Gamma$ and there are no square-islands in $y$ intersecting both $\Lambda$ and $\Z^2\setminus \Lambda$ then 
\begin{eqnarray}\label{equation: breaking conditionalmeasure into colours}
\mu_p\left([y]_{\Lambda}\big\vert [y]_{\Gamma\setminus\Lambda}\right)
&=&\mu\left([\phi(y)]_{\Lambda}\big\vert [\phi(y)]_{\partial\Lambda}\right)\prod_{v \in \Lambda}e^{g_p(\sigma_v(y))}.
\end{eqnarray}


Suppose $(y, y')\in \Delta_Y$ and
$$F=\{v \in \Z^2 ~:~ y_v\neq y'_v\}.$$

Note that if $y|_C$ is an infinite connected component consisting of square tiles then $y|_C= y'|_C$, that is, $C\cap F = \emptyset$. As in the proof of Proposition \ref{prop:island_mme}, there exists  a finite set $\Lambda\subset \Z^2$ such that $F \subset \Lambda$ and $y_i=y'_i$ is a blank tile for all $i \in \partial \Lambda$. Since $\mu$ is a Markov random field with uniform specification, \eqref{equation: breaking conditionalmeasure into colours} implies for all finite sets $\Gamma\supset \Lambda\cup \partial \Lambda$ that

\begin{equation}\label{equation: cocycle in terms of mini}
\frac{\mu_ p\left([y']_{\Gamma}\right)}{\mu_p\left([y]_{\Gamma}\right)}=  \prod_{v \in \Lambda}e^{g_p(\sigma_v(y'))-g_p(\sigma_v(y))}.
\end{equation}

It follows from \eqref{equation: cocycle in terms of mini} that

\begin{equation}\label{eq:M_p_ratio}
\frac{\mu_{p}([y']_\Gamma)}{\mu_{p}([y]_\Gamma)}= e^{M_{p}(y,y')}
\end{equation}
whenever $(y,y^{\p})\in \Delta_Y$ and $\Gamma$ satisfies the assumptions above. It follows that $M_p$ is the logarithm of the Radon-Nikodym cocycle for $\mu_p$.

\begin{prop}\label{prop:mu_p_MRF}
For any $p: \N \to  (0,1)$, the measure $\mu_p$ defined above is a shift-invariant Markov random field with Radon-Nikodym cocycle $M_p$.
\end{prop}
\begin{proof}
By the discussion above and \eqref{eq:M_p_ratio} it remains to check that $M_p$ is a  Markov cocycle.
This is a direct consequence of  Lemma \ref{lemma: squares decided by paths} below and the definition of $M_p$ in \eqref{eq:_M_p_def}.
\end{proof}

\begin{lemma}\label{lemma: squares decided by paths}
Let $(y,y'), (z, z')\in \Delta_Y$ and $F \subset \Z^2$ be a finite set so that
\begin{eqnarray*}
y|_{F\cup \partial F}=z|_{F\cup \partial F}&\text{ and }& y'|_{F\cup \partial F}= z'|_{F\cup \partial F},\\
y|_{F^c}= y'|_{F^c}&\text{ and }& z|_{F^c}= z'|_{F^c}.
\end{eqnarray*}

Then for $v \in \Z^2$  either
$$\psi_j(\sigma_v (y'))=\psi_j(\sigma_v (y)) \mbox{ and }  \psi_j(\sigma_v(z'))=  \psi_j(\sigma_v(z)) \text{ for }j=1,2$$
or
$$\psi_j(\sigma_v(y))=\psi_j(\sigma_v(z)) \mbox{ and } \psi_j(\sigma_v(y'))=\psi_j(\sigma_v(z')) \text{ for }j=1,2.$$
\end{lemma}
\begin{proof}
Suppose $\psi_j(\sigma_v (y')) \ne\psi_j(\sigma_v (y))$.
Then either $\psi_j(\sigma_v (y'))>0$ or $\psi_j(\sigma_v (y)) >0$. By symmetry considerations  we can  without loss of generality deal with the case that
$$\psi_1(\sigma_v (y')) \ne\psi_1(\sigma_v (y))= i >0,$$ so  $y|_{B_i+v}$ is an $i$-square-island of type $1$, and
 $y'|_{B_i+v}$ is not. We will prove that $\psi_1(\sigma_v(z))=i$ and $\psi_2(\sigma_v(z))=0$. The other implications can be obtained similarly. There are several cases to consider:

If $y'|_{B_i+v}$ is an $i$-square-island of type $2$ then $B_i+v \subset F$ so $z|_{B_i +v}$
is an $i$-square-island of type $1$  and so  $\psi_j(\sigma_v(y))=\psi_j(\sigma_v(z))$ for $j=1,2$.

Otherwise suppose that $y'|_{B_i+v}$ is not an $i$-square-island. A careful inspection of the possibilities  for
the pattern $y'|_{B_i+v}$ leads to the conclusion that there must exist a connected subset $F_0 \subset F$
and $(m_1,n_1), (m_2,n_2),(m_3,n_3) \in F_0$ such that $n_2-n_1=2i$ and  $\phi(z), \phi(y)$
satisfy the assumptions on  the point $x$  in the statement of Lemma \ref{lem:recover_square_island}.

We sketch some details:
First suppose there is a square  $S = (B_j +w )\subset \Z^2$  with $j >i$  so that $B_i +v \subset S$ and $y'|_S$ is an $j$-square-island.
Since $S \neq B_i +v$ are square shapes and $y|_{\partial((B_i+v)^c)}$ consists of border tiles, the set  $F$ completely contains at least one of the vertical sides of the square $B_i+v$. So in this case we can choose that side as $F_0$, choose $(m_1,n_2)$ and $(m_2,n_2)$ to be the locations of the corners for this side of $B_i+v$, and choose for instance $(m_3,n_3)=(m_1,n_1)$ .

Now suppose there is a square  $S = (B_j +w )\subset \Z^2$  with $j >i$  so that $S$ completely contains one of the sides of the square $B_i +v$ and $y'|_S$ is a $j$-square-island, but $B_i +v \not \subset S$.
Without loss of generality suppose $S$ contains the top side $T$ of the square $B_i+v$. Since $B_i+v\not\subset S$, $T$ does not intersect the top side of $S$ and thus $T\subset F$, and we can take $F_0=T$.

Otherwise, none of the sides of the square $B_i +v$ are contained in a square $S \subset \Z^2$ for which $y'|_S$ is a square-island. In this case we can choose
$$F_1 = \{ w \in B_i+v~:~ y'_w \mbox{ is a blank tile } \},$$
$$F_2 =  \left\{  w \in B_i+v ~:~  y'_w \mbox{ is a border tile and } y'_w\neq y_w\right\},$$
and $F_0 = F_1 \cup F_2$.

In any case it follows from Lemma \ref{lem:recover_square_island} that $z|_{B_i+v}$ is an $i$-square-island of type $1$, completing the proof.
\end{proof}

By Proposition \ref{prop:mu_p_MRF} above there is an  uncountable family of linearly independent shift-invariant Markov cocycles on $Y$  which have corresponding fully-supported shift-invariant Markov random fields. Since the space of Markov cocycles which come from some shift-invariant finite range interaction is a union of finite dimensional vector spaces, this further implies that there exists a shift-invariant Markov random field which is not Gibbs for any shift-invariant finite range interaction. Alternatively, note that for any Gibbs cocycle with some shift-invariant finite range interaction the magnitude of the cocycle at a particular homoclinic pair is at most linear in the size of set of sites at which the two configurations differ. We can choose $p:\N \to  (0,1)$ such that this does not happen.

A simple variation on the above construction yields topological Markov fields which are not sofic: Choose $f:\N \to  \{1,2\}$. For  each $f$ define a shift space $Y_f  \subset Y$, by the condition that each $i$-square island is of type $f(i)$. Each $Y_f$ is a topological Markov field so $\{Y_f\}_{ f:\N \to  \{1,2\}}$ is an uncountable family. Since the family of sofic subshifts is countable, typically $Y_f$  will not be sofic.

\section{Conclusion and further problems} In this paper we demonstrate the applicability of Markov cocycles to studying Markov random fields and Gibbs states with nearest neighbor interaction. In cases where the space of shift-invariant Markov cocycles is finite dimensional, these can act as a substitute for nearest neighbor interactions in providing a description of the specification ``using finitely many parameters''. We remark that in some recent work the formalism of Markov and Gibbs cocycles, the pivot property was used to generalise the strong version of the Hammersley-Clifford Theorem (Theorem \ref{thm: strong_hammersley_clifford}) beyond the safe symbol case in \cite{chandgotiahammcliff2014}.
We mention some questions which remain to be answered:

\subsection{Supports of Markov random fields} What characterizes the property of being the support of a shift-invariant Markov random field on the Cayley graph of $\mathbb{Z}^d$? We know it is equivalent to being a non-wandering nearest neighbor shift of finite type when $d= 1$. For finite graphs, it suffices to be a topological Markov field. In higher dimensions we wonder whether it is sufficient to have a topological Markov field with a shift-invariant probability measure supported on it. Also: Suppose a $\mathbb{Z}^d$ shift of finite type is the support of some shift-invariant Markov random field. Must it also be the support of a shift-invariant Gibbs measure for some nearest neighbor interaction?

\subsection{Algorithmic aspects}
Suppose we are given a nearest neighbor shift of finite type with the pivot property along with its globally allowed patterns on $\{0\}\cup \partial \{0\}$.
Is there an algorithm to determine the dimension of the shift-invariant Markov cocycles?
If so, is there a way to decide which of these cocycles have a corresponding fully supported shift-invariant probability measure on the subshift?
In case the subshift has a safe symbol, such an algorithm can be derived from the proof of the Hammersley-Clifford Theorem \cite{Preston} and also from Lemma 3.1 in \cite{dachiannahedescriptionofrandom}.

\subsection{Markov Random fields for other Models}
There are other models for which it would be interesting to get a good description of the Markov random Fields and Gibbs measures. One such family is the $r$-colorings of $\Z^d$ with $r \ge 2d+2$. Domino tilings of $\Z^2$ are another interesting model. For these examples we know the generalized pivot property holds, so the space of shift invariant Markov cocycles is finite dimensional. Our results were obtained specifically for the $3$-colored chessboards. Nevertheless, we believe some of the techniques we developed and applied in this paper can be useful in studying other systems.

\subsection{Changing the underlying graph}
In this paper we considered $X_r$ as a topological Markov field with respect to the standard Cayley graph of $\ZD$. One can also ask now how much does the underlying graph affect our results. For instance, what happens if we choose a different set of generators? For instance, adding generators corresponds to adding edges to the graph and allows for new nearest neighbour interactions.

Generalizing further, one can consider the analog of $X_r$ over a Cayley graph of some other finitely generated groups.
For instance when the underlying graph is a Cayley graph of $\Z$, we know that every shift-invariant Markov random field is Gibbs with a shift-invariant nearest neighbor interaction \cite{Markovfieldchain}, regardless of the topological Markov field.

\subsection{Mixing properties of subshifts and the dimension of the invariant Markov cocycles} In Section \ref{section: nonGibb},
we construct a subshift such that the dimension of the space of shift-invariant Markov cocycles is uncountable.
This subshift is topologically mixing but there are stronger  mixing properties which it does not exhibit (See \cite{boyle2010multidimensional} for a discussion of some mixing properties of $\ZD$-subshifts).
We wonder if there are natural mixing properties of a subshift which imply that the space of shift-invariant Markov cocycles is finite-dimensional.

\section*{Acknowledgements} We thank Prof.\!\! Brian Marcus, Prof.\!\! Ronnie Pavlov and Prof.\!\! Mike Boyle for many useful discussions. The first author is extremely grateful to Prof.\!\! Brian Marcus, his PhD advisor at the University of British Columbia under whose tutelage he has learnt all the ergodic theory that he currently knows. He will also like to thank the University of British Columbia for providing the necessary funding and opportunity to conduct this research. The second author would like to thank Prof.\!\! Brian Marcus, PIMS and the University of British Columbia for hosting him as a post-doctoral fellow during which period the problems in this paper were studied. We would like to thank the anonymous referees for a very careful reading and many helpful comments.

\bibliographystyle{abbrv}
\bibliography{MRF_3CB}
\end{document}